\documentclass[12pt]{article} 

\usepackage{fullpage}
\usepackage{amsmath,amssymb,amsthm}
\usepackage[all]{xy}

\newcommand{\Set}{\mathbf{Set}}
\newcommand{\Sets}{\ensuremath{\Set}}

\newcommand{\LH}{\mathbf{LH}}

\newcommand{\cat}[1]{\mathcal{#1}} 
\newcommand{\escat}[1]{\cat{#1}}

\newcommand{\Psh}[1]{\widehat{#1}}
\newcommand{\Shv}{\mathbf{Shv}}

\newcommand{\calC}{\ensuremath{\escat{C}}} 
\newcommand{\calE}{\ensuremath{\escat{E}}} 
\newcommand{\calF}{\ensuremath{\escat{F}}} 

\newcommand{\frakI}{\mathfrak{I}}

\newcommand{\opCat}[1]{\ensuremath{{#1}^{\mathrm{op}}}}

\newcommand{\Nat}{\ensuremath{\mathbb{N}}}

\newcommand{\Inv}{\mathrm{Inv}}
\newcommand{\bcu}{\mathrm{BCU}}

\newcommand{\below}[1]{(\downarrow\!{#1})}

\newcommand{\Rig}{\mathbf{Rig}}
\newcommand{\Ring}{\mathbf{Ring}}
\newcommand{\iRig}{\mathbf{iRig}}
\newcommand{\rRig}{\mathbf{rRig}}
\newcommand{\dLat}{\mathbf{dLat}}
\newcommand{\riRig}{\mathbf{riRig}}

\newcommand{\twopl}[2]{\ensuremath{\left\langle #1, #2 \right\rangle}}

\theoremstyle{plain}
\newtheorem{theorem}{Theorem}[section]
\newtheorem{corollary}[theorem]{Corollary}
\newtheorem{proposition}[theorem]{Proposition}
\newtheorem{lemma}[theorem]{Lemma}

\theoremstyle{definition}
\newtheorem{definition}[theorem]{Definition}

\newtheorem{remark}[theorem]{Remark}

\title{A representation theorem for integral rigs  and \\ its applications to residuated lattices}
\author{J.~L.~Castiglioni \and M. Menni \and W. J. Zuluaga Botero}

\begin{document}
\maketitle

\begin{abstract}
We prove that every integral rig in $\Sets$ is (functorially) the rig of global sections of a sheaf of really local integral rigs.  We also show that this representation result may be lifted to residuated integral rigs and then restricted to varieties of these. In particular, as a corollary, we obtain a representation theorem for pre-linear residuated join-semilattices in terms of totally ordered fibers. The restriction of this result to the level of MV-algebras coincides with the Dubuc-Poveda representation theorem.
\end{abstract}

\tableofcontents

\section{Introduction and background}

From theories of representation of rings by sheaves, due to
Grothendieck \cite{ega1}, Pierce \cite{Pierce67} and Dauns and Hoffman \cite{DaunsHofmann66},
general constructions of sheaves to universal algebras \cite{Davey73} evolved.
 All these representations where developed using  toposes of local homeos over  adequate
spaces. A concrete example of the method employed in the case of
bounded distributive lattices can be found in \cite{BrezuleanuDiaconescu69}.

In this line, Filipoiu and Georgescu find in \cite{FG} an equivalence
between MV-algebras and certain type of sheaves of MV-algebras over compact Hausdorff spaces.
In the same line, a presentation closer to the construction given by Davey in
\cite{Davey73}, is given by  Dubuc and Poveda in \cite{DubucPoveda2010}. They find an adjunction between MV-algebras and another version of MV-spaces. In \cite{FL} it is proposed a
third kind of representation for MV-algebras using fibers that are certain local MV-algebras.

As examples of representation by sheaves of other classes of
residuated structures, we can quote the Grothendieck-type duality
for Heyting algebras proposed in \cite{DNG} and the one given by Di Nola
and Leu\c stean in \cite{DiNolaLeustean2003} for BL-algebras.

   A more explicit use of topos theory is exemplified by the representation theorems for rings and lattices proved by Johnstone \cite{Johnstone77a} and Coste \cite{Coste79}. See also \cite{Cole}.

   The present work is motivated by the Dubuc-Poveda representation theorem for MV-algebras \cite{DubucPoveda2010, DubucErratum2010} and Lawvere's strategic ideas about the topos-theoretic analysis of coextensive algebraic categories hinted at in page 5 of \cite{Lawvere91} and also in \cite{Lawvere08}. For our main result we will borrow the notion of {\em really local rig} introduced in the unpublished \cite{LawvereEmailCoquand}.

The main result of this paper is a representation theorem for integral
rigs as internal really local integral rigs in toposes of sheaves
over bounded distributive lattices (Theorem~\ref{ThmMain}).
   We also show that this representation result may be lifted to residuated integral rigs and then restricted to varieties of these.
   In particular, as a corollary, we obtain a representation theorem for pre-linear residuated join-semilattices in terms of totally ordered fibers. The restriction of this result to the level of MV-algebras coincides with the Dubuc-Poveda representation theorem.

   We stress that our main results do not use topological spaces. Instead, we use the well-known equivalence between the topos of sheaves over a topological space $X$ and the topos of local homeos over $X$ in order to translate our results to the language of bundles.
This translation allows us to compare our results with related work.

   Sections~\ref{SecRigs} to~\ref{SecChineseReticulation} introduce the category of  integral rigs.
   Sections~\ref{SecReticInGroth} to \ref{SecFuntorialityOfRepresentation} recall the necessary background on topos theory and proves the main theorem, a representation result for integral rigs.
   Section~\ref{SecIntegralResiduatedRigs} shows how to apply our result to prove representation results for different categories of residuated join-semilattices.
   In order to make our results more accessible to a wider audience we explain, in Section~\ref{SecLH}, how to express our results in terms of local homeos over spectral spaces.
   Finally, in Section~\ref{SecMV}, we compare our corollary for MV-algebras with the motivating Dubuc-Poveda representation.

\section{The coextensive category of rigs}
\label{SecRigs}

   In this section we recall the notion of extensive category \cite{Lawvere91} an introduce the coextensive category of rigs \cite{Schanuel91, Lawvere08}.

\begin{definition}\label{DefExtensiveCategory}
A category $\calC$ is called {\em extensive} if it has finite coproducts and the canonical functors ${1 \rightarrow \calC/0}$ and ${\calC/X \times \calC/Y \rightarrow \calC/(X + Y)}$ are equivalences.
\end{definition}

   For example, any topos and the category of topological spaces are extensive. In contrast, an Abelian category is extensive if and only if it is trivial. If the opposite ${\opCat{\calC}}$ of a category $\calC$ is extensive then we may say that $\calC$ is {\em coextensive}. For example, the categories $\Ring$ and $\dLat$ of (commutative) rings and distributive lattices are coextensive.
   We will use the following characterization proved in \cite{CarboniLackWalters}.

\begin{proposition}\label{PropCharExtensivity}
A category with finite limits $\calC$ is extensive if and only if the following two conditions hold:
\begin{enumerate}
\item (Coproducts are disjoint.) For every $X$ and $Y$ the square below is a pullback
$$\xymatrix{
0 \ar[d]_-{!} \ar[r]^-{!} & Y \ar[d]^-{in_1} \\
X \ar[r]_-{in_0} & X + Y
}$$
\item (Coproducts are universal.) For every ${f:X \rightarrow 1 + 1}$, if the squares below are pullbacks
$$\xymatrix{
X_0 \ar[d] \ar[r] & X \ar[d]^-f & \ar[l] X_1 \ar[d] \\
1 \ar[r]_-{in_0} & 1 + 1 & \ar[l]^-{in_1}  1
}$$
then the top cospan ${X_0 \rightarrow X \leftarrow X_1}$ is a coproduct diagram.
\end{enumerate}
\end{proposition}

   The fact that $\Ring$  is coextensive rests on the fact that product decompositions of a ring are in correspondence with idempotents. Something analogous happens for $\dLat$. This analogy can be better explained by considering $\Ring$ and $\dLat$ as subcategories of the category of rigs introduced in \cite{Schanuel91}.

\begin{definition}\label{DefRig}
A {\em rig} is a structure ${(A, \cdot, 1, +, 0)}$ such that ${(A, \cdot, 1)}$ and ${(A, +, 0)}$ are commutative monoids and distributivity holds in the sense that ${a \cdot 0 = 0}$ and ${(a + b) \cdot c = a\cdot c + b\cdot c}$ for all ${a, b, c \in A}$.
\end{definition}

   As usual, we may avoid to write $\cdot$ in calculations and simply use juxtaposition.

   From Definition~\ref{DefRig} it is trivial to read off the presentation of an algebraic theory.
   We emphasize this because it is fundamental for our work that we can consider algebras in categories with finite products that are not necessarily the category $\Set$ of sets functions. In particular, we will consider rigs in toposes of sheaves over spectral spaces.

   For the moment let ${\Rig}$ be the algebraic category of rigs in $\Set$.
   As already suggested,  $\Rig$  embeds the  categories $\Ring$ and $\dLat$  and, moreover, shares with them the property of coextensivity \cite{Lawvere08}. The proof of this fact rests on the following concept.

\begin{definition}\label{DefBooleanElement}
An element $a$ in a rig $A$ is called {\em Boolean} if there exists an ${a'\in A}$ such that ${a + a' = 1}$ and ${a a' = 0}$.
\end{definition}

   Every Boolean element is idempotent and, in the case of rings, the converse holds.
   Just as in that case, Boolean elements in a rig correspond to its product decompositions. It follows from this that $\Rig$ is coextensive.

\begin{definition}[The canonical pre-order of a rig]
   Every rig $A$ is naturally equipped with a pre-order $\leq$ defined by ${x \leq y}$ if and only if there is a ${w \in A}$ such that ${w + x = y}$. Addition and multiplication are easily seen to be monotone with respect to this  pre-order.
\end{definition}

   The equation  ${1+1 = 1}$ holds in a rig $A$ if and only if addition is idempotent. In this case the canonical pre-order is a poset and coincides with that making the semilattice  ${(A, +, 0)}$ into a join semilattice.

   Let $\calE$ be a category with finite limits. For any rig $A$ in $\calE$  we define the subobject ${\Inv(A) \rightarrow A\times A}$ by declaring that the  diagram  below
$$\xymatrix{
\Inv(A) \ar[d] \ar[r]^-{!}  & 1 \ar[d]^-{1} \\
A\times A \ar[r]_-{\cdot} & A
}$$
is a pullback. The two projections ${\Inv(A) \rightarrow A}$ are mono in $\calE$ and induce the same subobject of $A$. Of course, the multiplicative unit ${1:1 \rightarrow A}$ always factors through ${\Inv A \rightarrow A}$. In particular, if $A$ is a distributive lattice then the factorization ${1:1 \rightarrow \Inv A}$ is an iso.

\begin{definition}\label{DefLocalMap}
A rig morphism ${f:A \rightarrow B}$ between rigs in $\calE$ is {\em local} if the following diagram
$$\xymatrix{
\Inv A \ar[d] \ar[r] & \Inv B \ar[d] \\
A \ar[r]_-f & B
}$$
is a pullback.
\end{definition}

If $\calE$ is a topos with subobject classifier ${\top:1 \rightarrow \Omega}$ then there exists a unique map ${\iota:A \rightarrow \Omega}$ such that the square below
$$\xymatrix{
 \Inv(A) \ar[d] \ar[r]^-{!}  & 1 \ar[d]^-{\top}\\
 A\ar[r]_-{\iota} & \Omega
}$$
is a pullback. It is well-known that the object $\Omega$ is an internal Heyting algebra and so, in particular, a distributive lattice. The basic properties of invertible elements imply that ${\iota:A \rightarrow \Omega}$ is a morphism of multiplicative monoids. The following definition is borrowed from \cite{LawvereEmailCoquand}.

\begin{definition}\label{DefReallyLocal}
The rig $A$ in $\calE$ is {\em really local} if ${\iota:A \rightarrow \Omega}$ is a morphism of rigs.
\end{definition}

In other words, $A$ is really local if ${\iota}$ is a map of additive monoids.
For example, a distributive lattice $D$ in $\Set$ is really local if and only if it is non-trivial and, for any ${x, y \in D}$, ${x \vee y = \top}$ implies ${x = \top}$ or ${y = \top}$.

The pullback  above Definition~\ref{DefReallyLocal} implies that if $A$ is really local then the morphism ${\iota:A \rightarrow \Omega}$ is local in the sense of Definition~\ref{DefLocalMap}. So, a rig $A$ is really local if and only if there is a (necessarily unique) local map of rigs ${A \rightarrow \Omega}$.

\section{The coextensive category of integral rigs}
\label{The coextensive category of integral rigs}

   As observed in Section~8 of \cite{Lawvere08}, distributive lattices are exactly the rigs satisfying ${x^2 = x}$ and ${1 + x = 1}$. The first equation says the multiplication is idempotent. It will be convenient to introduce some terminology for the second.

\begin{definition}\label{DefIntegralRig} A rig is called {\em integral} if ${1 + x = 1}$ holds.
\end{definition}

   If $A$ is an integral rig then, in particular, ${1 + 1 = 1}$ and so the canonical pre-order is a join-semilattice ${(A, +, 0)}$ such that  ${x \leq 1}$ for all $x$.
   Let us emphasize the following fact, stated in page~{502} of \cite{Lawvere08}.
   Say that $b$ {\em divides} $a$ in a rig if there exists $u$ such that ${b u = a}$.

\begin{lemma}\label{LemDividesImpliesAbove}
If  $b$ divides $a$ in an integral rig then ${a\leq b}$. In particular, ${x y \leq x}$ and ${x y \leq y}$ for any ${x, y \in A}$.
\end{lemma}
\begin{proof}
The equality ${b u = a}$ implies that ${a + b = b u + b = b(u+ 1) = b}$.
\end{proof}

   This is, of course,  a generalization of the implication ${(b \wedge u = a) \Rightarrow (a\leq b)}$ that holds in distributive lattices. Another relevant fact is the following.

\begin{lemma}\label{LemIntegralImpliesInvertiblesAreTrivial}
If $A$ is integral then the canonical ${1 \rightarrow \Inv(A)}$ is an iso.
\end{lemma}
\begin{proof}
If $x$ is invertible then $x$ divides $1$ and so ${1\leq x}$ by Lemma~\ref{LemDividesImpliesAbove}.
\end{proof}

   It follows that the problem of inverting elements in integral rigs is equivalent to the problem of forcing these elements to be $1$. For any subset ${S\subseteq A}$ let us write ${A \rightarrow A[S^{-1}]}$ for any solution to the universal problem of inverting all the elements of $S$. The notation is chosen so as to emphasize the relation with localizations in the case of classical commutative algebra.

   Fix an integral rig $A$ and let ${F\rightarrow A}$ be a multiplicative submonoid.
   For any ${x, y \in A}$ write ${x \mid_F y}$ if there exists ${w \in F}$ such that ${w x \leq y}$. It is easy to check that $\mid_F$ is a pre-order.

\begin{lemma}\label{LemSubMultiplicationsCanBeCollapsed}
If $A$ is integral and ${F\rightarrow A}$ is a multiplicative submonoid then the equivalence relation $\equiv_F$ determined by the pre-order $\mid_F$ is a congruence and the quotient ${A \rightarrow A/{\equiv_F}}$ has the universal property of ${A \rightarrow A[F^{-1}]}$.
\end{lemma}
\begin{proof}
It is easy to check that the pre-order $\mid_F$ is compatible with multiplication and then so is the induced equivalence relation $\equiv_F$.
To prove that $\mid_F$ is compatible with addition assume that ${x\mid_F x'}$ and ${y \mid_F y'}$ so that there are  ${w, v\in F}$ such that ${w x \leq x'}$ and ${v y \leq y'}$.
Then ${v w  (x + y) = v w  x + w v y \leq v x' + w y' \leq x' + y'}$. (Notice the importance of integrality in the last step.) This completes the proof that $\equiv_F$ is a congruence. To prove that the quotient ${A \rightarrow A/{\equiv_F}}$ collapses $F$ to $1$ notice that for any ${u \in F}$, ${u 1 \leq u}$ and so ${1 \mid_F u}$.
Finally assume that $B$ is integral and that ${f:A \rightarrow B}$ is a morphism that sends every element in $F$ to $1$. We claim that $f$ sends congruent elements in $A$ to the same element in $B$. It is enough to check that ${x\mid_F y}$ implies that ${f x \leq f y}$; but if ${w x \leq y}$ for some ${w \in F}$ then ${f x = (f w) (f x) = f(w x) \leq f y}$.
\end{proof}

   In particular, for any ${a \in A}$, the subset ${F = \{ a^n \mid n \in \Nat\} \subseteq A}$ is a multiplicative submonoid. In this case the universal ${A \rightarrow A[F^{-1}]}$ will be denoted by ${A \rightarrow A[a^{-1}]}$.

\begin{lemma}\label{LemLocalization} If $A$ is integral and ${a\in A}$ is idempotent then ${A \rightarrow A[a^{-1}]}$ may be identified with the quotient of $A$ by the congruence that identifies $x$ and $y$ exactly when ${a x = a y}$.
\end{lemma}
\begin{proof}
Consider the relation $\mid_F$ defined in Lemma~\ref{LemSubMultiplicationsCanBeCollapsed} for the particular case of the submonoid ${F = \{ a^n \mid n \in \Nat\} = \{1, a \} \subseteq A}$. Then ${x\mid_F y}$ if and only if ${x \leq y}$ or ${a x \leq y}$; and this holds if and only if ${a x \leq a y}$.
\end{proof}

   An element ${a}$ in a rig $A$ is called {\em nilpotent} if ${a^n = 0}$ for some ${n\in \Nat}$.

\begin{lemma}\label{LemTrivialLocalization}
For any integral rig $A$ and ${a \in A}$, ${0 = 1 \in A[a^{-1}]}$ if and only if $a$ is nilpotent.
\end{lemma}
\begin{proof}
One direction is trivial. For the other assume that ${0 = 1 \in A[a^{-1}]}$. Then ${1 \mid_F 0}$ in $A$ where ${F = \{ a^n \mid n \in \Nat\} \subseteq A}$. That is, there exists some ${n\in \Nat}$ such that ${a^n \cdot 1 \leq 0}$.
\end{proof}

   For certain special submonoids the construction of Lemma~\ref{LemSubMultiplicationsCanBeCollapsed} can be simplified. In order to explain this we introduce some notation. For any rig $A$ and subset ${S \subseteq A}$, we write  ${\below{S} \subseteq A}$ for the subset ${\{ x \mid (\exists s\in S)(x\leq s)\}\subseteq A}$.
   If ${a \in A}$ then we write ${\below{a}}$ instead of $\below{\{a\}}$. For example, ${\below{0} \subseteq A}$ consists of the elements that have a negative.

   If $A$ is integral then ${\below{a} \subseteq A}$ is closed under addition and multiplication.

\begin{definition}\label{DefStronglyIdempotent} An element ${a \in A}$ is called {\em strongly idempotent} if ${a x = x}$ for every ${x\leq a}$.
\end{definition}

   For instance, Boolean elements are strongly idempotent and, of course, every strongly idempotent element is idempotent. Also, every element in a distributive lattice is strongly idempotent.

\begin{lemma}\label{LemInvertingStronglyIdempotents} If $A$ is integral and ${a\in A}$ is strongly idempotent then addition and multiplication in ${\below{a}}$ may be extended to the structure of a rig, the function ${A \rightarrow \below{a}}$ that sends ${x \in A}$ to ${a x \in \below{a}}$ is a rig morphism and has the universal property of ${A \rightarrow A[a^{-1}]}$.
\end{lemma}
\begin{proof}
If $a$ is strongly idempotent then $a$ is the unit for the restricted multiplication in $\below{a}$ and the function ${A \rightarrow \below{a}}$ that sends ${x \in A}$ to ${a x \in \below{a}}$ clearly preserves addition and multiplication,  and it  sends $a$ to ${1 \in \below{a}}$, so there exists a unique morphism ${A[a^{-1}] \rightarrow \below{a}}$ such that the following diagram
$$\xymatrix{
A \ar[r] \ar[rd] & A[a^{-1}] \ar[d] \\
& \below{a}
}$$
commutes. We prove that ${A[a^{-1}] \rightarrow \below{a}}$ is bijective using the description of the domain given in Lemma~\ref{LemLocalization}. It is clearly injective. On the other hand, if ${x \leq a}$ then ${a x = x}$ because $a$ is strongly idempotent.
\end{proof}

   Let ${\iRig \rightarrow \Rig}$ be the full subcategory of integral rigs in $\Set$.
   As explained in \cite{Lawvere08}, the inclusion   ${\iRig \rightarrow \Rig}$ has a right adjoint and so it follows that $\iRig$ is coextensive. Since we have not proved that $\Rig$ is coextensive we give below a direct proof that $\iRig$ is. The result is not needed for our representation theorem  but it is an essential part of the conceptual path that leads to it.

\begin{proposition}\label{PropIRigIsCoextensive}
The category $\iRig$ is coextensive.
\end{proposition}
\begin{proof}
We use the dual of Proposition~\ref{PropCharExtensivity}. For any pair of integral rigs $A$, $B$, the pushout of the projections ${A \leftarrow A\times B \rightarrow B}$ is the terminal object because  the element ${(1, 0) \in A\times B}$ is sent to $1$ in $A$ and to $0$ in $B$ so ${1 = 0}$ must hold in the pushout.

   The initial distributive lattice $2$ is also initial in $\iRig$. So consider a map ${g:2\times 2 \rightarrow A}$ and assume the squares below
$$\xymatrix{
2 \ar[d] & \ar[l]_-{\pi_0} 2\times 2 \ar[d]^-g \ar[r]^-{\pi_1} & 2 \ar[d] \\
A_0 &  \ar[l] A \ar[r] & A_1
}$$
are pushouts. The element ${(1, 0)}$ in ${2\times 2}$ has complement ${(0, 1)}$ and so $g$ is determined by ${a_0 = g(1, 0)}$ and also by ${a_1 = g(0, 1)}$. The universal property of the pushouts implies that ${A \rightarrow A_0}$ coincides with ${A \rightarrow A[a_0^{-1}]}$ and that ${A \rightarrow A_1}$ coincides with ${A \rightarrow A[a_1^{-1}]}$.
   So, to complete the proof, it is enough to show that if $a$ is complemented in $A$ then the canonical ${A \rightarrow A[a^{-1}] \times A[b^{-1}] = \below{a}\times \below{b}}$ is an iso, where $b$ is the complement of $a$. Surjectivity is easy because if ${x\leq a}$ and ${y\leq b}$ then ${x+y \in A}$ is sent to ${(a x, b y) = (x, y) \in \below{a}\times \below{b}}$.
   To prove injectivity assume that ${u, v \in A}$ are such that ${(a u, b u) = (a v, b v)}$.
Then the following calculation
\[ u = (a +b)u = a u  + b u = a v + b v = (a + b) v = v \]
completes the proof.
\end{proof}

   We will need a further general fact about localizations.
	 Again, let $A$ be an integral rig and ${F \rightarrow A}$ be a multiplicative submonoid.
   Consider $F$ equipped with the partial order inherited from $A$. As a poset, $F$ is cofiltered because for every ${x, y \in F}$, ${x y \leq x}$ and ${x y \leq y}$ by integrality. The assignment that sends ${x\in F}$ to ${A[x^{-1}]}$ in $\iRig$ determines a functor ${\opCat{F} \rightarrow \iRig}$ because if ${x\leq y}$ in $F$ then there exists a unique map ${A[y^{-1}] \rightarrow A[x^{-1}]}$ such that the triangle on the left below
$$\xymatrix{
A \ar[r] \ar[rd] & A[y^{-1}] \ar[d] \ar[r] & A[F^{-1}] \\
 & A[x^{-1}] \ar[ru]
}$$
commutes. Also, the universal map ${A \rightarrow A[F^{-1}]}$ determines maps ${A[y^{-1}] \rightarrow A[F^{-1}]}$ and ${A[x^{-1}] \rightarrow A[F^{-1}]}$ and, moreover, the universal property of ${A \rightarrow A[y^{-1}]}$ implies that the triangle on the right above commutes. In other words, we have a cocone from the functor ${\opCat{F} \rightarrow \iRig}$ to ${A[F^{-1}]}$. We claim that it is colimiting. To prove this consider a cocone $g$ from the functor ${\opCat{F} \rightarrow \iRig}$ to an integral rig $B$. Then there exists a diagram as on the left below
$$\xymatrix{
A \ar[r] \ar[rd] & A[y^{-1}] \ar[d] \ar[r]^{g_y} & B &&
  A \ar[rd]_-{g'} \ar[r] & A[F^{-1}] \ar[d]^-{\overline{g}} \\
 & A[x^{-1}] \ar[ru]_-{g_x} & &&
  & B
}$$
which shows that we have a map ${g':A \rightarrow B}$ that inverts every element of $F$. Hence, there exists a unique ${\overline{g}:A[F] \rightarrow B}$ such that the diagram on the right above commutes. For each ${x \in F}$, the universal property of ${A \rightarrow A[x^{-1}]}$ implies that the following diagram
$$\xymatrix{
A[x^{-1}] \ar[rd]_-{g_x} \ar[r] & A[F^{-1}] \ar[d]^-{\overline{g}} \\
 & B
}$$
commutes. Moreover, this $\overline{g}$ is the unique one with this property, so the above claim is true.
   Alternatively, if we denote the colimit of the functor ${\opCat{F} \rightarrow \iRig}$ by ${\varinjlim_{x \in \opCat{F}} A[x^{-1}]}$ then the above argument implies the following.
	
\begin{lemma}\label{LemColomitFilters}
For any multiplicative submonoid ${F \rightarrow A}$ there exists a unique isomorphism ${A[F^{-1}] \rightarrow \varinjlim_{x \in \opCat{F}} A[x^{-1}]}$ such that the following diagram
$$\xymatrix{
A \ar[d] \ar[r] & A[F^{-1}] \ar[d] \\
A[x^{-1}] \ar[r] &  \varinjlim_{x \in \opCat{F}} A[x^{-1}]
}$$
commutes for every ${x \in F}$.
\end{lemma}

\section{Chinese remainder and reticulation}
\label{SecChineseReticulation}

   In this section we prove a generalization of the Dubuc-Poveda `pushout-pullback' lemma \cite{DubucPoveda2010} and give an explicit simple description of the left adjoint to the inclusion ${\dLat \rightarrow \iRig}$ of the category of distributive lattices into that of integral rigs. These results play an important role in the proof of our representation result.

\begin{lemma}\label{LemIntegralityBasic}
If $A$ is an integral rig then,   ${(x + y)^m  \leq x^m + x y +  y^m}$ for any ${x, y \in A}$ and  any ${m \in \Nat}$.
Therefore, for any ${m, n \in \Nat}$, ${(x + y)^{m n} \leq x^m + y^n}$.
\end{lemma}
\begin{proof}
The result is clear for ${m = 0}$ and ${m = 1}$, so let ${m \geq 2}$.
Since addition is idempotent,
\[ (x + y)^m = \sum_{i = 0}^m \binom{m}{i} x^{m - i} y^i = \sum_{i = 0}^m  x^{m- i} y^i \]
and since ${m\geq 2}$ and $A$ is integral,
\[ (x + y)^m = x^m + \left(\sum_{i =1}^{m -1} x^{m- i} y^i  \right) + y^m  =
x^m + x y \left(\sum_{i =1}^{m -1} x^{m- i-1} y^{i-1}  \right) + y^m
\leq x^m + x y + y^m \]
so the first part of the result holds. We now claim that ${(x + y)^m \leq x + y^m}$ for every ${m \in \Nat}$. It certainly holds for ${m = 0}$. Now, if ${m \geq 1}$ then
\[ (x + y)^m \leq x^m + x y + y^m = x(x^{m - 1} + y) + y^m \leq x + y^m \]
holds. Commutativity implies that  ${(x + y)^m \leq x^m + y}$ for every ${m\in \Nat}$.
Finally,
\[ (x + y)^{m n} = ((x + y)^m)^n \leq (x^m + y)^{n} \leq x^m + y^n \]
completes the proof.
\end{proof}


      Let $A$ be an integral rig and let ${a, b \in A}$. If ${f:A \rightarrow R}$ in $\iRig$ is such that ${f a = 1}$ then ${f(a + b) =  1 + f b = 1}$. So there exists a unique map ${A[(a + b)^{-1}] \rightarrow R}$ such that the composite ${A \rightarrow A [(a + b)^{-1}] \rightarrow R}$ equals $f$. Similarly, if $f$ inverts $b$.
   On the other hand, if ${f:A \rightarrow R}$  inverts ${a b}$ then there is a canonical cospan ${A[a^{-1}] \rightarrow R \leftarrow A[b^{-1}]}$.

\begin{lemma}[The pushout-pullback lemma]\label{LemPBPO}
If $A$ is integral then for every ${a, b \in A}$ the canonical maps make the following diagram commute
$$\xymatrix{
A[(a + b)^{-1}] \ar[d] \ar[r] & A[b^{-1}] \ar[d] \\
A[a^{-1}] \ar[r] & A[(a b)^{-1}]
}$$
and, moreover, the square is both a pullback and a pushout.
\end{lemma}
\begin{proof}
The fact that the square commutes may be checked by pre-composing with the universal ${A \rightarrow A[(a + b)^{-1}]}$. The fact that the square is a pushout follows because the cospan ${A[a^{-1}] \rightarrow A[(a b)^{-1}]\leftarrow A[b^{-1}]}$ is actually a coproduct.

To prove that the diagram is a pullback it seems convenient to introduce some notation.
For any ${x\in A}$, we write ${(x\bmod a)}$ for the associated element in ${A[a^{-1}]}$; and similarly for $b$, ${a+b}$ and ${a b}$. For example, the left vertical map of the diagram in the statement sends ${(z \bmod (a+b))}$ to ${(z\bmod a)}$.
    Consider the actual pullback ${A[a^{-1}] \leftarrow P \rightarrow A[b^{-1}]}$ and the induced canonical morphism ${A[(a + b)^{-1}] \rightarrow P}$. We need to show that this map is bijective.

    To prove that ${A[(a + b)^{-1}] \rightarrow P}$ is surjective let ${(x\bmod a)}$ and ${(y\bmod b)}$ be such that ${(x \bmod a b) = (y \bmod a b)}$. This means that that there exists an ${m\in\Nat}$ such that  ${(a b)^m x \leq y}$ and ${(a b)^m y \leq x}$.
   Consider now ${((a^m x + b^m y) \bmod (a + b))}$. Clearly, ${a^m x \leq a^m x + b^m y}$ and, on the other hand,  calculate
\[ a^m(a^m x + b^m y) = a^{2 m} x + (a b)^m y \leq  a^{2 m} x +  x = (a^{2 m} + 1) x = x \]
to conclude  that ${((a^m x + b^m y) \bmod a) = (x \bmod a)}$. An analogous calculation shows that ${((a^m x + b^m y) \bmod b) = (y \bmod b)}$.

  To prove that ${A[(a + b)^{-1}] \rightarrow P}$ is injective it is enough to show that the canonical ${A[(a + b)^{-1}] \rightarrow A[a^{-1}] \times A[b^{-1}]}$ is injective. So let ${(u \bmod (a + b))}$ and ${(v \bmod (a + b))}$ be such that they are sent to the same pair in ${A[a^{-1}] \times A[b^{-1}]}$. That is: ${(u \bmod a) = (v \bmod a)}$ and ${(u \bmod b) = (v \bmod b)}$.  In turn, this means that there are ${k, l \in \Nat}$ such that
\[ a^k u \leq v \quad a^k v \leq u \quad\quad  b^l u \leq v \quad b^l v \leq u \]
so, using Lemma~\ref{LemIntegralityBasic}, we can calculate
\[ (a + b)^{k l} u \leq (a^k + b^l) u \leq a^k u + b^l u \leq v + v \leq v \]
and, similarly, ${(a + b)^{k l} v \leq u}$. Therefore, ${(u \bmod (a + b)) = (v \bmod (a + b))}$.
\end{proof}

    Fix an integral rig $A$ and define an auxiliary pre-order $\preceq$ on $A$ by declaring that ${x \preceq y}$ holds if and only if there exists an ${m\in \Nat}$ such that ${x^m \leq y}$. Since multiplication is monotone with respect to $\leq$, $\preceq$ is indeed a pre-order. This pre-order determines, as usual, an equivalence relation $\sim$ on $A$. That is, ${x\sim y}$ if and only if both ${x \preceq y}$ and ${y\preceq x}$.

\begin{lemma}\label{LemLattification}
If $A$ is an integral rig the relation $\sim$ is a rig congruence and the quotient ${\eta_A:A \rightarrow A/{\sim}}$ is universal from $A$ to the inclusion ${\dLat \rightarrow \iRig}$. Moreover, the map ${\eta_A:A \rightarrow A/{\sim}}$ is local.
\end{lemma}
\begin{proof}
To prove that $\sim$ is a congruence it is enough to show that addition and multiplication are compatible with $\preceq$.
Assume that ${u\preceq v}$ and ${x\preceq y}$, so that there are ${m , n\in \Nat}$ such that ${u^m \leq v}$ and ${x^n \leq y}$. Lemma~\ref{LemIntegralityBasic} implies that ${(u + x)^{m n} \leq u^m + x^n \leq v + y}$, so ${u + x  \preceq v + y}$.
On the other hand, ${(u x)^{\max(m, n)} = u^{\max(m, n)} x^{\max(m, n)} \leq u^m x^n \leq v y }$, so ${u x  \preceq v y}$. Therefore, $\sim$ is indeed a congruence.

   We now prove that $A/{\sim}$ is a lattice. Clearly, the surjection ${\eta = \eta_A: A \rightarrow A/{\sim}}$ implies that ${1 + u = 1}$ for every ${u\in A/{\sim}}$. To prove that multiplication is idempotent let ${x \in A}$ and observe that ${(\eta x)(\eta x) = \eta(x^2)}$.
But ${x^2 \sim x}$ so ${(\eta x)^2 = \eta(x^2) = \eta x}$.

   To prove the universal property let ${f:A \rightarrow D}$ be a rig morphism with $D$ a distributive lattice.
   If ${x\preceq y}$ then there is an ${m\in \Nat}$ such that ${x^m \leq y}$ and, since multiplication is idempotent in $D$, ${f x = f(x^m) \leq f y}$. Hence, ${x \sim y}$ implies ${f x = f y}$.

   To prove ${\eta_A:A \rightarrow A/{\sim}}$ is local assume that ${\eta x = \eta 1}$.
   Then ${1 \preceq x}$ and so ${1 \leq x}$.
\end{proof}

   Let us denote the resulting left adjoint by ${L:\iRig \rightarrow \dLat}$ and the associated unit by ${\eta_A = \eta:A \rightarrow L A}$. This unit and its codomain ${L A}$ may be referred to as the {\em reticulation} of the rig $A$.

   The next result shows that reticulations interact well with localizations.

\begin{lemma}\label{LemOverlineAndReticulation}
For any integral rig $A$ and ${x \in A}$  both squares below
$$\xymatrix{
A \ar[d] \ar[r]^-{\eta}  & L A \ar[d]           && A \ar[d] \ar[r]^-{\eta} & L A \ar[d] \\
A[x^{-1}] \ar[r] & (L A)[(\eta x)^{-1}] && A[x^{-1}] \ar[r]_-{\eta} & L(A[x^{-1}])
}$$
are pushouts in ${\iRig}$ where the bottom map on the left above is the unique rig map that makes the square commute; and the right square is the obvious naturality square.
\end{lemma}
\begin{proof}
The fact that the left square is a pushout follows from the universal property of the localization ${L A \rightarrow (L A)[(\eta x)^{-1}]}$ and the fact that ${A \rightarrow A[x^{-1}]}$ is epi. To complete the proof it is enough to show that the map ${A[x^{-1}] \rightarrow (L A)[(\eta x)^{-1}]}$ is universal from ${A[x^{-1}]}$ to distributive lattices.
So let $D$ be a distributive lattice and  ${f:A[x^{-1}] \rightarrow D}$ be a morphism of rigs.
By the universal property of ${\eta:A \rightarrow L A}$ there is a unique ${f_1:L A \rightarrow D}$ such that the square on the left below
$$\xymatrix{
A \ar[d] \ar[r]^-{\eta} & L A \ar[d]^-{f_1} &&
   A[x^{-1}] \ar[rd]_-{f} \ar[r] & (L A)[(\eta x)^{-1}] \ar[d]_-{f_2}^{f_3} & \ar[ld]^-{f_1} \ar[l] L A \\
A[x^{-1}] \ar[r]_-f     & D                 &&
                                    & D
}$$
commutes. By the pushout property,  there exists a unique map ${f_2:(L A)[(\eta x)^{-1}] \rightarrow D}$ such that the two triangles on the right above commute. In particular, the left triangle commutes. Finally, if there exists an ${f_3:(L A)[(\eta x)^{-1}] \rightarrow D}$ such that the left triangle above commutes then the right triangle also commutes because ${\eta:A \rightarrow L A}$ is epi. Hence, ${f_2 = f_3}$.
\end{proof}

\section{Reticulation in Grothendieck toposes}
\label{SecReticInGroth}

    Let $\calE$ be a topos and let ${\iRig(\calE)}$ be the category of intregral rigs in $\calE$. We can consider the full subcategory ${\dLat(\calE) \rightarrow \iRig(\calE)}$ and wonder about the existence of a left adjoint. If the left adjoint exists then we say that $\calE$ {\em has reticulations (of integral rigs)}. In Lemma~\ref{LemLattification} we proved that $\Set$ has reticulations. It should be possible to internalize this result to an elementary $\calE$ with a natural numbers object (A2.5.1 in \cite{elephant}). Also, it is tempting to conjecture that the existence of the left adjoint ${\iRig(\calE) \rightarrow \dLat(\calE)}$ implies that the unit of the adjunction is epi and local. Since we have not been able to prove this conjecture we lift the adjunction (and the properties of its unit)  from $\Set$ to an arbitrary Grothendieck topos.
    It is convenient to show first that inverse images of geometric morphisms preserve reticulations.

\begin{lemma}\label{LemUpperStarsAppliedToReticulations}
Let ${F:\calF\rightarrow \calE}$ be a geometric morphism. If $A$ is an integral rig in $\calE$ and ${\eta:A\rightarrow L A}$ is its reticulation then ${F^* \eta:F^* A \rightarrow F^* (L A)}$ is the reticulation (in $\calF$) of the integral rig ${F^* A}$. Moreover, if $\eta$ is epi and local then so is ${F^* \eta}$.
\end{lemma}
\begin{proof}
Denote the unit of ${F^* \dashv F_*}$ by ${\alpha:id_{\calF} \rightarrow F_* F^*}$. Since $\alpha$ is natural and $F^*$ preserves products, ${\alpha_A:A \rightarrow F_* (F^* A)}$ is a rig map. If $D$ is a distributive lattice in $\calF$ and ${f:F^* A \rightarrow D}$ is a rig map then the transposition ${(F_* f) \alpha:A \rightarrow F_* D}$ is a composite of rig maps with codomain ${F_* D}$ in ${\dLat(\calE)}$. Hence, there exists a unique rig morphism ${f':L A \rightarrow F_* D}$ in $\calE$ such that the diagram on the left below
$$\xymatrix{
A  \ar[d]_-{\alpha} \ar[r]^-{\eta} & L A  \ar[d]^-{f'} &&
   \ar@/_3.5pc/[dd]_-{id} F^* A \ar[d]^-{F^* \alpha} \ar[r]^-{F^* \eta} & F^* (L A) \ar[d]^-{F^* f'} &&
     F^* A \ar[rd]_-{f} \ar[r]^-{F^* \eta} & F^* (L A) \ar[d]^-{g} \\
F_* (F^* A) \ar[r]_-{F_* f} & F_* D &&
   F^* (F_* (F^* A)) \ar[d]^-{\beta} \ar[r]_-{F^* (F_* f)} & F^* (F_* D) \ar[d]^-{\beta} && & D \\
 & && F^* A \ar[r]_-{f} & D
}$$
commutes. Then the middle diagram above commutes, where ${\beta}$ is the counit of ${F^* \dashv F_*}$. So the existence part of the universal property holds. To prove the uniqueness part assume that ${g:F^* (L A) \rightarrow D}$ in $\calF$ is a rig map  such that the right diagram above commutes and check that the transposition of $g$ equals $f'$.

   The second part of the statement follows because ${F^*:\calE \rightarrow \calF}$ is a finite-limit-preserving left adjoint.
\end{proof}

    We now consider how to lift reticulations from $\Set$ to presheaf toposes.
    So let $\calC$ be a small category and consider and integral rig $P$ in the presheaf topos ${\Psh{\calC}}$. In other words, $P$ is a functor ${\opCat{\calC} \rightarrow \iRig}$ so we can consider the composite
$$\xymatrix{
\opCat{\calC} \ar[r]^-{P} & \iRig \ar[r]^-{L} & \dLat
}$$
which is a distributive lattice ${L P}$ in $\Psh{\calC}$. Moreover, for every ${t:b \rightarrow c}$ in $\calC$, the following diagram
$$\xymatrix{
P c \ar[d]_-{P t} \ar[r]^-{\eta_{P c}} & L (P c) \ar[d]^-{L(P t)} \\
P b \ar[r]_-{\eta_{P b}} & L (P b)
}$$
commutes, where ${\eta_{P c}: P c \rightarrow L(P c)}$ is the reticulation of the integral rig ${P c}$ in $\Set$. So we obtain a natural transformation ${\eta_P:P \rightarrow L P}$ such that, for each object $c$ in $\calC$, ${\eta_{P c}}$ is a rig morphism. In other words, we have obtained a rig morphism ${P \rightarrow L P}$ in $\Psh{\calC}$.

\begin{lemma}\label{LemReticulationsInPresheafCats}
The transformation ${\eta_P:P \rightarrow L P}$ above is the reticulation of $P$ in $\Psh{\calC}$. Moreover, ${\eta_P}$ is epi and local as a rig map.
\end{lemma}
\begin{proof}
Let $Q$ be a distributive lattice in $\Psh{\calC}$ and let ${f:P \rightarrow Q}$ be a rig morphism. Then, for each $c$ in $\calC$, there exists a unique rig map ${f'_c: L(P c) \rightarrow Q c}$ in $\Set$ such that the  diagram on the left below
$$\xymatrix{
P c \ar[rd]_-{f_c} \ar[r]^-{\eta_{P c}} & L(P c) \ar[d]^-{f'_c} &&
   L(P c) \ar[d]^-{f'_c} \ar[r]^-{L(P t)} & L(P b) \ar[d]^-{f'_b} \\
 & Q c &&
   Q c \ar[r]_-{Q t} & Q b
}$$
commutes. Also, for  every ${t:b \rightarrow c}$ in $\calC$, it is easy to check that the diagram on the right above commutes; just precompose with ${\eta_{P c}}$. So we obtain a rig map ${f':L P \rightarrow Q}$ in $\Psh{\calC}$ such that ${f' \eta_P = f}$. Since each ${\eta_{P c}: P c \rightarrow L(P c)}$ is surjective in $\Set$, ${\eta_P:P \rightarrow L P}$ is epi in $\Psh{\calC}$, so $f'$ is the unique map satisfying the above equality.

   It remains to show that ${\eta_P:P \rightarrow L P}$ is local; but this follows because each rig morphism ${\eta_{P c}: P c \rightarrow L(P c)}$ is local in $\Set$, and limits in presheaf toposes are calculated `as in $\Set$'.
\end{proof}

    We can now prove lift reticulations to bounded toposes over $\Set$.

\begin{proposition}\label{PropReticulationsInGrothendieckToposes}
Grothendieck toposes have reticulations. Moreover, these are epi and local.
\end{proposition}
\begin{proof}
Let ${\calF = \Shv(\calC, J)}$ be a Grothendieck topos. There is a topos inclusion ${F:\calF \rightarrow \Psh{\calC}}$ so the counit ${\beta:F^* F_* \rightarrow id_{\calF}}$ is an iso.
Now let $A$ be an integral rig in $\calF$. By Lemma~\ref{LemReticulationsInPresheafCats} there is an epi and local reticulation ${\eta:F_* A \rightarrow L(F_* A)}$ in ${\Psh{\calC}}$
and, by Lemma~\ref{LemUpperStarsAppliedToReticulations}, the integral rig ${F^* (F_* A)}$ in $\calF$, has an epi and local reticulation ${F^* \eta:F^* (F_* A) \rightarrow F^* (L (F_* A))}$. Since ${\beta:F^* (F_* A) \rightarrow A}$ is an iso of rigs, the result follows.
\end{proof}

\section{Really local integral  rigs}

   Fix a topos $\calE$ and recall (\cite{LawvereEmailCoquand} and Section~\ref{SecRigs}) that a rig $A$ is called {\em really local} if  there  is a (necessarily unique) local  morphism ${A \rightarrow \Omega}$ of rigs.
   A basic exercise in the internal logic of toposes shows the following.

\begin{lemma}\label{LemPresentationOfReallyLocalRigs}
 The rig $A$ is really local if and only if the following sequents hold
\[
\begin{array}{rcl}
0 \in \Inv(A) & \vdash & \perp_{\ \ } \\
(x + y) \in \Inv(A) & \vdash_{x, y} & x \in \Inv(A) \quad  \vee \quad y \in \Inv(A) \\
x \in \Inv(A) \quad \vee \quad y \in \Inv(A) & \vdash_{x, y}  & (x + y) \in \Inv(A) \\
\end{array}
\]
in the internal logic of $\calE$.
\end{lemma}
\begin{proof}
For example, the composite ${\xymatrix{1 \ar[r]^-0 & A \ar[r]^-{\iota} & \Omega}}$ equals the bottom element ${\perp:1 \rightarrow \Omega}$ if and only if the following rectangle
$$\xymatrix{
0 \ar[d]_-{!} \ar[d] \ar[r]^-{!} & \Inv(A) \ar[d] \ar[r]   & 1 \ar[d]^-{\top} \\
1 \ar[r]_-{0}                    & A \ar[r]_-{\iota} & \Omega
}$$
is a pullback. Since the right square is a pullback by definition, the rectangle is a pullback if and only if the left square is a pullback.
\end{proof}

   Notice that if $A$ was a ring then the first two sequents in the statement of the lemma would say that $A$ is local in the usual sense. In other words: ``The preservation of addition is a strengthening, possible for positive quantities, of the usual notion of localness (which on truth values was only an inequality)" \cite{LawvereEmailCoquand}.

  Naturally, really local integral rigs  have a simpler characterization.

\begin{lemma}\label{LemReallyLocalIntegral}
An integral rig is really local if and only if the following sequents hold
\[
\begin{array}{rcl}
0 = 1 & \vdash & \perp_{\ \ } \\
x + y = 1 & \vdash_{x, y} & x = 1 \quad  \vee \quad y = 1 \\
\end{array}
\]
in the internal logic of $\calE$.
\end{lemma}
\begin{proof}
   Recall that the multiplicative unit ${1:1 \rightarrow A}$ factors through ${\Inv(A) \rightarrow A}$ and, if $A$ is integral, the factorization ${1:1 \rightarrow \Inv(A)}$ is an iso (Lemma~\ref{LemIntegralImpliesInvertiblesAreTrivial}). In other words, if $A$ is integral then the following square
$$\xymatrix{
1 \ar[d]_-{1} \ar[r] & 1 \ar[d]^-{\top} \\
A \ar[r]_-{\iota} & \Omega
}$$
is a pullback.
\end{proof}

   For any rig $A$ in $\calE$  define the object of ${\bcu(A) \rightarrow A\times A}$ of {\em Binary Covers of the Unit} by declaring the square below
$$\xymatrix{
\bcu(A) \ar[d] \ar[r] & 1 \ar[d]^-{1}  \\
A\times A \ar[r]_-{+} & A
}$$
to be a pullback in $\calE$. If $A$ is integral then there is a unique ${\lambda:A \rightarrow \bcu(A)}$ such that the diagram on the left below
$$\xymatrix{
A \ar[d]_-{\twopl{!}{id}} \ar[r]^-{\lambda} & \bcu(A) \ar[d] &&
   A \ar[d]_-{\twopl{id}{!}} \ar[r]^-{\rho} & \bcu(A) \ar[d] \\
1\times A \ar[r]_-{1\times id} & A\times A &&
   A\times 1 \ar[r]_-{id\times 1} & A\times A
}$$
commutes. Symmetrically, there is a unique ${\rho:A \rightarrow \bcu(A)}$ such that the diagram on the right above commutes. (In the case ${\calE = \Set}$, ${\bcu(A) = \{ (x, y) \mid x + y = 1 \}}$, ${\lambda x = (1, x)}$ and ${\rho x = (x, 1)}$ for ${x\in A}$.)
 We can now reformulate Lemma~\ref{LemReallyLocalIntegral} as follows.

\begin{lemma}\label{LemCharReallyLocalInCoherentCat}
An integral rig $A$ in $\calE$ is really local if and only if the diagram below
$$\xymatrix{
0 \ar[r]^-{!} & 1 \ar[r]<+1ex>^-{0} \ar[r]<-1ex>_-{1} & A
}$$
is an equalizer and  ${[\lambda, \rho]:A + A \rightarrow \bcu(A)}$ is epi.
\end{lemma}


   We now characterize really local integral rigs in presheaf toposes.
   Let $\calC$ be a small category and ${\Psh{\calC}}$ be the associated presheaf topos.
   Since the theory of integral rigs is clearly algebraic, an integral rig internal to $\Psh{\calC}$ is simply a functor ${A:\opCat{\calC} \rightarrow \iRig}$.

\begin{lemma}\label{LemReallyLocalRigsInPresheafToposes}
An integral rig $A$ in $\Psh{\calC}$ is really local if and only if the integral rig ${A c}$ is really local in $\Set$ for every $c$ in $\calC$.
\end{lemma}
\begin{proof}
The fork in Lemma~\ref{LemCharReallyLocalInCoherentCat} is an equalizer in $\Psh{\calC}$ if and only if for every ${c\in \calC}$, ${0 \not= 1 \in A c}$.
   Also, for each $c$ in $\calC$, ${(\bcu(A)) c = \{ (x, y) \in (A c) \times (A c) \mid x + y = 1 \in A c \}}$. Hence, the map ${[\lambda, \rho]:A + A \rightarrow \bcu(A)}$ is epi in ${\Psh{\calC}}$ if and only if ${[\lambda_c, \rho_c]:(A c) + (A c) \rightarrow (\bcu(A))c}$ is epi in $\Set$ for each $c$ in $\calC$. In other words, for each ${c\in \calC}$ and  every ${x, y \in A c}$, ${x + y = 1}$ implies ${x = 1}$ or ${y = 1}$.
\end{proof}

\section{Coherent localic toposes}

   Although Proposition~\ref{PropReticulationsInGrothendieckToposes} shows that reticulations exist in all Grothendieck toposes we are going to be mainly interested in coherent localic toposes or, equivalently, toposes of sheaves over spectral spaces. See D3.3.14 in \cite{elephant}.

   Let $D$ be a distributive lattice seen as a coherent category (A1.4 in \cite{elephant}).
   Its {\em coherent coverage} (A2.1.11(b) loc.~cit.) is the function that sends each ${d \in D}$ to the set of finite families ${(d_i \leq d \mid i\in I)}$ such that ${\bigvee_{i\in I} d_i = d}$. (These will be called {\em covering} families or simply {\em covers}.) The resulting topos of sheaves will be denoted by ${\Shv(D)}$. It is well-known that ${\Shv(D)}$  is equivalent to the topos of sheaves on the locale corresponding to the frame of ideals of $D$ (see C2.2.4(b) loc.~cit.).

   Binary covers ${a \vee b = d}$ of ${d\in D}$ will play an important role because in order to check that a presheaf ${P:\opCat{D}\rightarrow \Set}$ is a sheaf, it is enough to check the sheaf condition for binary covers.
   The empty family covers an object ${d\in D}$ if and only if ${d= \bot}$. So, for any ${X \in \Shv(D)}$, ${X \bot}$ is terminal in $\Set$.

   Since the theory of integral rigs is algebraic, it is well-known that an integral rig in $\Shv(D)$ is a functor ${\opCat{D} \rightarrow \iRig}$  such that the composite presheaf ${\opCat{D} \rightarrow \iRig \rightarrow \Set}$ is a sheaf. In this section we characterize really local integral rigs in ${\Shv(D)}$ which requires a little extra effort because the theory of really local rigs is not algebraic. Specifically, we will need explicit descriptions of finite colimits in ${\Shv(D)}$.

\begin{lemma}[Epis in $\Shv(D)$]\label{LemEpisInShvD} A map ${f:X \rightarrow Z}$ is epi in ${\Shv(D)}$ if and only if for every ${d\in D}$ and ${z \in Y d}$, there exists a cover ${\bigvee_{i\in I} d_i  = d}$ such that ${z\cdot d_i}$ is in the image of the function ${f_{d_i}:X d_i \rightarrow Z d_i}$.
\end{lemma}
\begin{proof}
This is an instance of Corollary~{III.7.5} in \cite{maclane2}.
\end{proof}

   An ${f:X \rightarrow Z}$ as in Lemma~\ref{LemEpisInShvD} is sometimes called `locally surjective'.

   Consider now the initial object ${0 \in \Shv(D)}$. As any other sheaf, it must satisfy that ${0\bot}$ is terminal in $\Set$. It is easy to check that ${0 d = 0 \in \Set}$ for any ${d > \bot}$.

\begin{lemma}[Binary coproducts in $\Shv(D)$]\label{LemCoprosInShvD} For every $X$, $Y$ in $\Shv(D)$, the coproduct ${X + Y}$ may be defined by
\[ (X + Y) d = \{ (a, b, x, y) \mid a \vee b = d , a\wedge b = \bot, x\in (X a) , y \in (Y b) \} \]
and, for any ${(a, b, x, y)\in (X + Y) d}$,
\[ ((X + Y)(c\leq d))(a, b, x, y)  = (a, b, x, y) \cdot c = (a\wedge c, b\wedge c, x\cdot (a\wedge c), y\cdot (b\wedge c)) \]
where ${x\cdot c = (X(c\leq d)) x \in X c}$ and ${y\cdot c = (Y(c\leq d)) y \in Y c}$.
\end{lemma}
\begin{proof}
It is easy to check that, defined as above, ${X + Y}$ is indeed a presheaf on $D$. To prove that it is a sheaf assume that ${c_0 \vee c_1 = d}$ and let
${ (a_0, b_0, x_0, y_0) \in (X + Y) c_0}$ and ${(a_1, b_1, x_1, y_1) \in (X + Y) c_1}$ be a compatible family. That is, for any ${c \leq c_0 \wedge c_1}$,
\[ (a_0, b_0, x_0, y_0) \cdot c = (a_1, b_1, x_1, y_1) \cdot c \]
which, may be instructive to draw as follows
$$
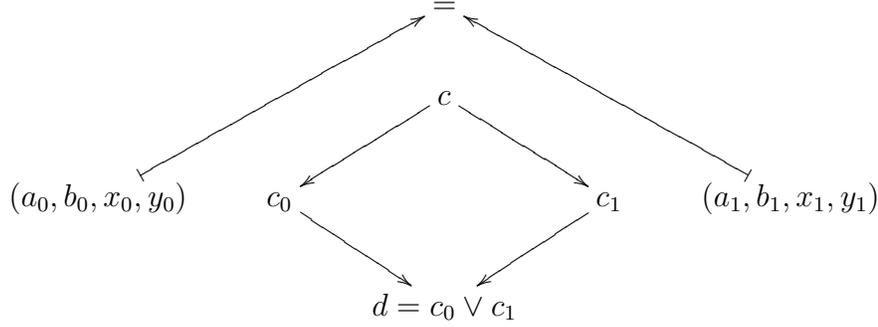
\xymatrix{
 & & = \\
                     &             & \ar[ld] c \ar[rd] \\
(a_0, b_0, x_0, y_0) \ar@{|->}[rruu] & c_0 \ar[rd] &                   & \ar[ld] c_1 & \ar@{|->}[lluu] (a_1, b_1, x_1, y_1) & \\
                     &             & d = c_0 \vee c_1
}$$
and means that the following equations
\[ a_0 \wedge c = a_1 \wedge c \quad
   b_0 \wedge c = b_1 \wedge c \quad
   x_0 \cdot (a_0 \wedge c) = x_1 \cdot (a_1\wedge c) \quad
   y_0 \cdot (b_0 \wedge c) = y_1 \cdot (b_1\wedge c)
    \]
hold. We have to show that there is a unique amalgamation.
First, we claim that  the pair ${a = a_0 \vee a_1}$, ${b = b_0 \vee b_1}$ is a partition of $d$.
Of course, ${a\vee b = d}$ so we need only prove that ${a\wedge b = \bot}$.
Since
\[ a \wedge b = (a_0 \wedge b_0) \vee (a_0 \wedge b_1) \vee (a_1\wedge b_0) \vee (a_1 \wedge b_1) = (a_0 \wedge b_1) \vee (a_1\wedge b_0) \]
we are left to show that ${a_0 \wedge b_1 = \bot = a_1\wedge b_0}$.
Taking ${c = c_0 \wedge c_1}$, compatibility means
\[ a_0 \wedge (c_0 \wedge c_1) = a_1 \wedge (c_0 \wedge c_1) \quad
   b_0 \wedge (c_0 \wedge c_1) = b_1 \wedge (c_0 \wedge c_1) \]
\[   x_0 \cdot (a_0 \wedge (c_0 \wedge c_1)) = x_1 \cdot (a_1\wedge (c_0 \wedge c_1)) \quad
   y_0 \cdot (b_0 \wedge (c_0 \wedge c_1)) = y_1 \cdot (b_1\wedge (c_0 \wedge c_1))
    \]
which simplifies to
\[ a_0 \wedge c_1 = a_1 \wedge c_0  \quad
   b_0 \wedge c_1 = b_1 \wedge c_0  \]
\[   x_0 \cdot (a_0 \wedge  c_1) = x_1 \cdot (a_1\wedge c_0) \quad  y_0 \cdot (b_0 \wedge  c_1) = y_1 \cdot (b_1\wedge c_0 )   \]
so \[ a_0 \wedge b_1 = a_0 \wedge b_1 \wedge c_1 = a_1 \wedge b_1 \wedge c_0 = \bot\wedge c_0 = \bot \]
and analogously,  ${a_1 \wedge b_0 = \bot}$. This completes the proof that ${a, b \leq d}$ is a partition of $d$.

   Consider now ${x_0 \in X a_0}$ and ${x_1 \in X a_1}$. Since
\[    x_0 \cdot (a_0 \wedge a_1) = x_0 \cdot (a_0 \wedge c_1) \cdot (a_0 \wedge a_1) =
x_1 \cdot (a_1\wedge c_0) \cdot (a_0 \wedge a_1) = x_1  \cdot (a_0 \wedge a_1) \]
 and $X$ is a sheaf,  there exists a unique ${x \in X a = X(a_0 \vee a_1)}$ such that ${x \cdot a_0 = x_0}$ and ${x\cdot a_1 = x_1}$.
   Similarly, there exists a unique ${y \in Y b = Y(b_0 \vee b_1)}$ such that ${y\cdot b_0 = y_0}$ and ${y\cdot b_1 = y_1}$. Altogether, we have obtained an element ${ (a, b, x, y) \in (X+ Y) d}$. Moreover,
\[ (a, b, x, y) \cdot c_0 = (a \wedge c_0, b\wedge c_0, x \cdot (a\wedge c_0), y\cdot (b\wedge c_0)) = (a_0, b_0, x_0, y_0) \]
because \[ a \wedge c_0 = (a_0 \vee a_1) \wedge c_0 = (a_0 \wedge c_0) \vee (a_1 \wedge c_0) = a_0 \vee (a_0 \wedge c_1) = a_0 \] and,  analogously,  ${b\wedge c_0 = b_0}$.

   Assume now that ${(a', b', x', y') \in (X+ Y) d}$ is such that ${(a', b', x', y')\cdot c_0 = (a_0, b_0, x_0, y_0)}$ and ${(a', b', x', y')\cdot c_1 = (a_1, b_1, x_1, y_1)}$. This means that
\[ a'\wedge c_0 = a_0 \quad b'\wedge c_0 = b_0 \quad x'\cdot (a' \wedge c_0) = x_0 \quad
y' \cdot (b' \wedge c_0) = y_0 \]
\[ a'\wedge c_1 = a_1 \quad b'\wedge c_1 = b_1 \quad x'\cdot (a' \wedge c_1) = x_1 \quad
y' \cdot (b' \wedge c_1) = y_1 \]
so ${a = a_0 \vee a_1 = (a' \wedge c_0) \vee (a' \wedge c_1) = a' \wedge (c_0 \vee c_1) = a' \wedge d = a'}$ and, similarly, ${b = b'}$. The first and third columns above imply ${x'\cdot a_0 = x_0}$ and ${x'\cdot a_1 = x_1}$ so $x'$ must be $x$. The second and fourth columns imply that ${y' = y}$. This completes the proof that ${X + Y}$ is a sheaf.

   Let $\bullet$ denote the unique element in ${Y\bot}$ and, for each $d$ in $D$, consider the function ${X d \rightarrow (X+ Y) d}$ that sends ${x\in X d}$ to ${(d, \bot, x, \bullet)}$.
   It is not difficult to check that this definition is natural in $d$ and so, determines a map ${X \rightarrow X + Y}$ in $\Shv(D)$. Of course, we have the analogue ${Y \rightarrow X + Y}$ and we claim that together they form a coproduct diagram.
   To prove this consider a cospan ${f:X \rightarrow Z\leftarrow Y:g}$. Assume that we have a map ${h:X \rightarrow Z}$ such that the following diagram
$$\xymatrix{
X \ar[rd]_-f \ar[r] & X + Y \ar[d]^-h & \ar[l] \ar[ld]^-g Y \\
 & Z
}$$
commutes. For and ${d \in D}$ and ${(a, b, x, y) \in (X + Y) d}$ we have that
\[ (h_d(a, b, x, y))\cdot a = h_a((a, b, x, y)\cdot a) = h_a(a, \bot, x \cdot a, \bullet) = f_a x \]
and, similarly, ${(h_d(a, b, x, y))\cdot b = g_b y}$. As $Z$ is a sheaf, ${h_d(a, b, x, y)}$ is uniquely determined by the last two equations. It is straightforward to check that if we define ${h_d(a, b, x, y)}$ in this way then the resulting family of functions is natural in $d$.
\end{proof}

   In particular, ${(1 + 1) d = \{ (a, b) \mid a\vee b = d, a\wedge b = \bot \}}$. Roughly speaking, ${1 + 1}$  is the `object of partitions' of $D$.

\begin{proposition}\label{PropCharReallyLocal}
An integral rig $A$ in ${\Shv(D)}$ is really local if and only if, for every ${d\in D}$ the following two conditions hold:
\begin{enumerate}
\item If ${A d = 1}$ then ${d = \bot}$. (Equivalently, ${0 = 1 \in A d}$ implies ${d = \bot}$.)
\item For all ${u, v \in A d}$ such that ${u + v = 1}$, there is a cover ${a \vee b = d}$ such that ${u\cdot a = 1}$ and ${v\cdot b = 1}$.
\end{enumerate}
\end{proposition}
\begin{proof}
This is just Lemma~\ref{LemCharReallyLocalInCoherentCat} in the particular case of the topos $\Shv(D)$. Indeed, if we let ${E \rightarrow 1}$ be the equalizer of ${0, 1:1 \rightarrow A}$ then ${E d}$ is terminal or initial in $\Set$ depending on whether ${0 = 1}$ in ${A d}$ or not. So, the unique map ${0 \rightarrow E}$ is an iso if and only if the first item in the statement holds.

   To complete the proof we show that the second condition in the statement is equivalent to epiness of the canonical map ${[\lambda, \rho]:A + A \rightarrow \bcu(A)}$  defined in Lemma~\ref{LemCharReallyLocalInCoherentCat}.

   Assume first that the second condition of the present statement holds. To prove that ${[\lambda, \rho]:A + A \rightarrow \bcu(A)}$ is epi, fix ${d\in D}$ and ${(u, v) \in (\bcu(A))d}$. Then ${u + v = 1 \in A d}$ and, by hypothesis, there is a cover ${a\vee b = d}$ such that ${u\cdot a = 1 \in A a}$ and ${v\cdot b = 1 \in A b}$.
   So ${(u, v) \cdot a = (1, v\cdot a) = \lambda_a (v\cdot a) = [\lambda, \rho]_a (in_{0, a} (v\cdot a))}$ where ${in_0:A \rightarrow A + A}$ is the `left' coproduct inclusion and, similarly, ${(u, v) \cdot b = [\lambda, \rho]_b (in_{1, b} (u\cdot b))}$.
   Hence, we have proved that the map ${[\lambda, \rho]:A + A \rightarrow \bcu(A)}$ is locally surjective.

   For the converse  assume that ${[\lambda, \rho]:A + A \rightarrow \bcu(A)}$ is locally surjective. Again, let ${d\in D}$ and ${u, v\in A d}$ be such that ${u + v = 1 \in A d}$.
By assumption there is a cover ${\bigvee_{i\in I} d_i = d}$ and, for each ${i\in I}$, an element ${(a_i, b_i, x_i, y_i)\in (A + A) d_i}$ such that the equation ${[\lambda, \rho]_{d_i}   (a_i, b_i, x_i, y_i) = (u, v) \cdot d_i = (u\cdot d_i , v\cdot d_i)}$ holds.

   If we let ${a = \bigvee_{i\in I} a_i}$ and ${b = \bigvee_{i\in I} b_i}$ then clearly, ${a \vee b = d}$. Now calculate:
\[    (u\cdot d_i \cdot a_i, v\cdot d_i \cdot a_i) = (u\cdot d_i , v\cdot d_i) \cdot a_i = ([\lambda, \rho]_{d_i}   (a_i, b_i, x_i, y_i)) \cdot a_i = \]
\[= [\lambda, \rho]_{a_i}  ( (a_i, b_i, x_i, y_i)\cdot a_i) =  [\lambda, \rho]_{a_i} (a_i, \bot, x_i, 1) = [\lambda, \rho]_{a_i} (in_{0, a_i} x_i) = \lambda_{a_i} x = (1, x_i) \]
and conclude that ${ u \cdot a \cdot a_i = u\cdot d_i \cdot a_i =  1}$ for every ${i\in I}$.
Since $A$ is a sheaf, ${u\cdot a = 1 \in A a}$.
A similar calculation shows that ${v \cdot b = 1 \in A b}$.
\end{proof}

\section{Principal subobjects}
\label{SecPrincipalSubobjecs}

   The content of the present section has probably been considered elsewhere but we have not found it. Let $D$ be a distributive lattice and let ${\Psh{D}}$ be the topos of presheaves on $D$.

\begin{definition}\label{DefPrincipalSub}
For any $X$ in  $\Psh{D}$, a subobject ${U \rightarrow X}$ is called {\em principal} if for every ${d\in D}$ and ${x \in X d}$ there exists a largest ${c\leq d}$ such that ${x\cdot c \in U c}$.
\end{definition}

   We are going to be mainly interested in principal subobjects in ${\Shv(D)}$.

\begin{lemma}\label{LemPrincipalSubobjectsAreClosed} If $X$ is a sheaf and ${u:U \rightarrow X}$ is principal then $U$ is also a sheaf.
\end{lemma}
\begin{proof}
 Let ${a \vee b = d}$ and let ${x \in U a}$ and ${y \in U b}$ be a compatible family. Then $x$ and $y$ form also a compatible family for $X$ and hence, there exists a unique ${z \in X d}$ such that ${z \cdot a = x}$ and ${z \cdot b = y}$. By hypothesis, there exist a largest ${e \leq d}$ such that ${z \cdot e \in U e}$. Then ${a \leq e}$ and ${b \leq e}$, so ${a \vee b = d \leq e}$. Therefore, ${d = e}$ and $U$ is a sheaf.
\end{proof}

   Let $\Lambda$ in $\Psh{D}$ be the object that sends ${d\in D}$ to ${\below{d}}$. For any ${d\in D}$, ${x \in \Lambda d}$ and ${c \leq  d}$, we have that ${x \cdot c = x \wedge c \in \Lambda c}$.

\begin{lemma}\label{LemLambdaIsAsheaf}
The presheaf $\Lambda$ is a sheaf.
\end{lemma}
\begin{proof}
Let ${a \vee b = d}$ and let ${x \in \Lambda a = \below{a}}$ and ${y \in \Lambda b = \below{b}}$ be a compatible family, so that ${x \cdot (a\wedge b) = y\cdot (a \wedge b)}$.
That is, ${x \wedge b = y \wedge a}$. Then ${x \vee y \in \below{d}}$,
\[ (x\vee y) \cdot a =  (x\vee y) \wedge a = (x \wedge a) \vee (y \wedge a) = x \vee (x\wedge b) = x \]
and, similarly, ${(x\vee y) \cdot b = y}$. Finally, assume that ${z \in \Lambda d}$ is such that ${z\cdot a = x}$ and ${z \cdot b = y}$. That is, ${z \wedge a = x}$ and ${z \wedge b = y}$, so ${(z \wedge a) \vee (z\wedge b) = x \vee y}$ and hence, \[x \vee y = z \wedge (a \vee b) = z \wedge d = z \]
which completes the proof that $\Lambda$ is a sheaf.
\end{proof}

   There is an obvious point ${\top:1 \rightarrow \Lambda}$ that, at stage $d$, is the top element of ${\Lambda d = \below{d}}$.
   Of course, there exists a unique ${\iota:\Lambda \rightarrow \Omega}$ such that the following diagram
$$\xymatrix{
1 \ar[d]_-{\top} \ar[r]^-{id} & 1 \ar[d]^-{\top} \\
\Lambda \ar[r]_-{\iota} & \Omega
}$$
is a pullback in ${\Shv(D)}$. The map ${\iota_d:\Lambda d \rightarrow \Omega d}$ sends ${x \leq d}$ to the principal ideal ${\below{x} \in \Omega d}$.

\begin{lemma}\label{LemClassifierOfPrincipal}
 The subobject ${\top:1 \rightarrow \Lambda}$ in ${\Shv(D)}$ classifies principal subobjects.
\end{lemma}
\begin{proof}
Let ${u:U \rightarrow X}$ be mono in ${\Shv(D)}$. Its classifying map ${\chi:X \rightarrow  \Omega}$ sends $x$ to the ideal ${\chi x = \{c \leq d \mid x\cdot c \in U c\}}$. Hence, ${\chi}$ factors through  ${\iota:\Lambda \rightarrow \Omega}$ if and only if ${u:U \rightarrow X}$ is principal.
\end{proof}

   In other words, principal subobjects are those  whose characteristic maps are valued in principal ideals.

\begin{lemma}\label{LemIotaIsArigMap}
The object $\Lambda$ is a really local distributive lattice in ${\Shv(D)}$ with top element ${\top:1 \rightarrow \Lambda}$. Moreover, the mono ${\iota:\Lambda \rightarrow \Omega}$ is a rig morphism.
\end{lemma}
\begin{proof}
It is clear that ${\Lambda d}$ is a lattice for each $d$ and that the action of the presheaf preserves this structure. So $\Lambda$ is indeed a distributive lattice in ${\Shv(D)}$. To complete the proof it is enough to check that ${\iota_d:\below{d}\rightarrow \Omega d}$ is a morphism of rigs. It is clear that ${\iota_d:\below{d}\rightarrow \Omega d}$ preserves top and bottom element. It is also clear that for ${a, b \leq d}$, \[ \iota_d (a \wedge b) = \below{(a\wedge b)} = \below{a} \cap \below{b} = (\iota_d a) \wedge (\iota_d b) \in \Omega d \]
and, moreover, ${\below{(a\vee b)}}$ is the least ideal in $D$  containing the lower set ${\below{a} \cup \below{b}}$ in $D$ so ${\below{a} \vee \below{b} = \below{(a\vee b)} \in \Omega d}$. Hence, ${\iota_d (a\vee b) = (\iota_d a) \vee (\iota_d b) \in \Omega d}$.
   That is, we have a rig morphism ${\iota:\Lambda \rightarrow \Omega}$. Since it is mono, it is local and hence, ${\Lambda}$ is a really local rig.
\end{proof}

   Surely, ${\Lambda}$ is the internal distributive lattice of compact open sublocales of the internal locale determined by $\Omega$, but we will not need this fact.

\begin{lemma}\label{LemDefAcutelyLocal}
If $X$ is an integral rig in ${\Shv(D)}$ then the following are equivalent.
\begin{enumerate}
\item The rig $X$ is really local and ${1:1\rightarrow X}$ is principal.
\item The rig $X$ is really local and for every ${d\in D}$ and ${x \in X d}$ there exists a largest ${c\leq d}$ such that ${x\cdot c = 1 \in X c}$.
\item There is a local morphism of rigs ${X \rightarrow \Lambda}$.
\end{enumerate}
Moreover, in case the above holds, the map ${X \rightarrow \Lambda}$ is unique.
\end{lemma}
\begin{proof}
The first two items are clearly equivalent so it is enough to show that the first and third items are equivalent. Let ${\chi:X \rightarrow \Omega}$ be the characteristic map of ${1:1 \rightarrow X}$. If the subobject ${1:1 \rightarrow X}$ is principal, $\chi$ factors through the mono ${\Lambda \rightarrow \Omega}$ and so there is a (unique) map ${X \rightarrow \Lambda}$ such that the square on the left below
$$\xymatrix{
1 \ar[d] \ar[r]^-{id} & 1 \ar[d]^-{\top} \ar[r]^-{id} & 1 \ar[d]^-{\top} \\
X \ar@/_1pc/[rr]_-{\chi} \ar[r] & \Lambda \ar[r] & \Omega
}$$
is a pullback. Since $X$ is really local, ${\chi:X \rightarrow \Omega}$ is a morphism of rigs. As the inclusion ${\Lambda\rightarrow\Omega}$ is also a morphism of rigs then the factorization ${X \rightarrow \Lambda}$ is also a morphism of rigs and it is local because the left square above is a pullback. On the other hand, if there is a local morphism of rigs ${X \rightarrow \Lambda}$ then the left square above is a pullback. So the subobject ${1 \rightarrow X}$ is principal and the rectangle is a pullback. Therefore the composite ${X \rightarrow \Lambda \rightarrow \Omega}$ morphism of rigs equals $\chi$ and hence  $X$ is really local. Finally, since ${X \rightarrow \Lambda}$ is the characteristic map of the principal subobject ${1:1 \rightarrow X}$, it is unique.
\end{proof}

\section{The category of representations}

   In this section we will write ${\top:1 \rightarrow \Lambda_D}$ for the classifier of principal subobjects in ${\Shv(D)}$, with $D$ a distributive lattice.
     This more explicit notation is necessary because we need to consider different lattices at the same time.

\begin{definition}\label{DefRepresentation}
A {\em representation (of an integral rig)} is a pair ${(D, X)}$ consisting of a distributive lattice $D$ and an integral rig $X$ in ${\Shv(D)}$ satisfying the equivalent conditions of Lemma~\ref{LemDefAcutelyLocal}.
\end{definition}

   It follows from Lemma~\ref{LemDefAcutelyLocal} that every representation ${(D, X)}$ determines a local map of rigs that we denote by ${\chi:X \rightarrow \Lambda_D}$.
   If necessary, to avoid confusion, we may denote it by ${\chi_X:X \rightarrow \Lambda_D}$.

   We now define a category ${\frakI}$ whose objects are representations in the sense above.
To describe the arrows in $\frakI$ first recall that any rig map ${f:D \rightarrow E}$ between distributive lattices induces a functor ${f_*:\Shv(E) \rightarrow \Shv(D)}$ that sends $Y$ in ${\Shv(E)}$ to the composite
$$\xymatrix{
\opCat{D} \ar[r]^-{\opCat{f}} & \opCat{E} \ar[r]^-{Y} & \Set
}$$
which lies in ${\Shv(D)}$. See Theorem~{VII.10.2} in \cite{maclane2}. Moreover, the functor ${f_*}$ is the direct image of a geometric morphism ${\Shv(E) \rightarrow \Shv(D)}$ so it sends integral rigs in the domain to integral rigs in the codomain.

   The map ${f:D \rightarrow E}$ also determines a morphism ${f':\Lambda_D \rightarrow f_* \Lambda_E}$ of lattices in ${\Shv(D)}$ such that for each ${d\in D}$, ${f'_d:\Lambda_D d = \below{d} \rightarrow (f_* \Lambda_E) d = \Lambda_E (f d) = \below{f d}}$ sends ${a\leq d}$ to ${f a \leq f d}$. Now that we have made this explicit, it is convenient to write ${f:\Lambda_D \rightarrow f_* \Lambda_E}$ and forget about the $f'$ notation.

   We can now define the maps in $\frakI$. For representations ${(D, X)}$ and ${(E, Y)}$, a map ${(D, X) \rightarrow (E, Y)}$ in $\frakI$ is a pair ${(f, \phi)}$ with ${f:D \rightarrow E}$ and ${\phi:X \rightarrow f_* Y}$ rig maps such that the following diagram
$$\xymatrix{
X \ar[d]_-{\chi} \ar[r]^-{\phi} & f_* Y \ar[d]^-{f_* \chi} \\
\Lambda_D \ar[r]_-{f} & f_* \Lambda_E
}$$
commutes in ${\Shv(D)}$. We emphasize that ${f:D \rightarrow E}$ is a morphism in ${\dLat(\Set)}$ and ${\phi:X \rightarrow f_* Y}$ is a morphism  in ${\iRig(\Shv(D))}$.

        If ${(f, \phi):(D, X) \rightarrow (E, Y)}$ and ${(g, \gamma):(C, W) \rightarrow (D, X)}$ are maps in $\frakI$ then
$$\xymatrix{ W \ar[r]^-{\gamma} & g_* X \ar[r]^-{g_* \phi} & g_*(f_* Y) = (f g)_* Y}$$
so we can define the {\em composite} ${ (f, \phi) (g, \gamma):(C, W) \rightarrow (E, Y)}$ as the pair ${(f g, (g_* \phi) \gamma)}$.

\begin{lemma} Composites of maps in $\frakI$ are maps in $\frakI$.
\end{lemma}
\begin{proof}
We use the notation above. First notice that the following rectangle
$$\xymatrix{
W \ar[r]^-{\gamma} \ar[d]_-{\chi} & g_* X \ar[d]^-{g_* \chi} \ar[r]^-{g_* \phi} & g_*(f_* Y) \ar[d]^-{g_* (f_* \chi)} \ar[r]^-{id} & (f g)_* Y \ar[d]^-{(f g)_* \chi} \\
\Lambda_C \ar[r]_-g & g_* \Lambda_D \ar[r]_-{g_* f} & g_*(f_* \Lambda_E) \ar[r]_-{id} & (f g)_* \Lambda_E
}$$
commutes in ${\Shv(C)}$. Indeed, the left-most square commutes because ${(g, \gamma)}$ is an map in $\frakI$; the middle square commutes because ${(f, \phi)}$ is a map in $\frakI$ and ${g_*}$ is a functor; and the right-most square commutes because  ${g_* f_* = (f g)_*:\Shv(E) \rightarrow \Shv(C)}$. It remains to show that the bottom composite is the map ${\Lambda_C \rightarrow (f g)_* \Lambda_E}$ induced by ${f g:C \rightarrow D}$; but this is straightforward.
\end{proof}

   For any object ${(D, X)}$ in $\frakI$ it is easy to check that the pair ${(id_D:D \rightarrow D, id_X:X \rightarrow X)}$ is a map in $\frakI$ that will be called the {\em identity}  on ${(D, X)}$.

\begin{lemma}
Composition of maps as defined above determines a category $\frakI$.
\end{lemma}
\begin{proof}
It remains to show that composition is associative and that identities are neutral with respect to it. We leave the details for the reader.
\end{proof}

   For each ${(D, X)}$ in $\frakI$ define ${\Gamma(D, X) = X \top}$ and, for ${(f, \phi):(D, X) \rightarrow (E, Y)}$ in $\frakI$, define ${\Gamma (f, \phi) = \phi_{\top}: D \top \rightarrow (f_* Y) \top = Y (f\top) = Y\top}$. It is very easy to prove that this induces a functor ${\Gamma:\frakI \rightarrow \iRig}$.

\section{The representation of integral rigs}

    Recall that ${L:\iRig \rightarrow \dLat}$ is the left adjoint to the (nameless) full inclusion in the opposite direction and that the unit of the associated adjunction is denoted by ${\eta:Id \rightarrow L}$.

    Fix and integral rig $A$ in $\Set$ and its reticulation ${\eta:A \rightarrow L A}$.
    Let ${\Psh{L A} = \Set^{\opCat{(L A)}}}$ be the topos of presheaves on ${L A}$. We now explain how the rig $A$ in $\Sets$ determines a integral rig $\overline{A}$ in $\Psh{L A}$.

\begin{lemma}\label{LemLocalizationReticulation}
For any ${x, y \in A}$, if ${\eta x \leq \eta y \in L A}$ then there is a (necessarily unique) map ${A[y^{-1}] \rightarrow A[x^{-1}]}$ making the following diagram
$$\xymatrix{
A \ar[r] \ar[rd] & A[y^{-1}] \ar[d] \\
                 & A[x^{-1}]
}$$
commute, where the horizontal and diagonal arrows are the respective localizations.
Hence, if ${\eta x = \eta y}$ then ${A[x^{-1}]}$ is canonically iso to ${A[y^{-1}]}$.
\end{lemma}
\begin{proof}
If ${\eta x \leq \eta y}$ in ${L A}$,  there is a ${w \in A}$ such that ${\eta w + \eta x  = \eta y}$. That is, ${\eta (w + x) = \eta y}$ so there exists an ${m\in \Nat}$ such that ${(w + x)^m \leq y}$ and hence ${x^m \leq (w + x)^m \leq y}$. But then, using the notation of Lemma~\ref{LemPBPO},  ${1 = (x^m \bmod x) \leq (y \bmod x) \in A[x^{-1}]}$.
\end{proof}

   It follows that the assignment that sends ${\eta x \in L A}$ to ${A[x^{-1}]}$ is well defined. Moreover, if ${\eta x \leq \eta y}$ then we have a canonical morphism ${A[y^{-1}] \rightarrow A[x^{-1}]}$ and it easy to check that we obtain a functor ${\opCat{(L A)} \rightarrow \iRig}$. In other words, we obtain an integral rig ${\overline{A}}$ in the presheaf topos ${\Psh{L A}}$.

   Notice that if $A$ is a distributive lattice then $\overline{A}$ coincides with the classifer ${\Lambda_{L A}}$ of principal subobjects in ${\Shv(L A)}$. So the following result may be seen as a generalization of Lemma~\ref{LemLambdaIsAsheaf}.

\begin{lemma}\label{LemNewIntegralSheaf} For any integral rig $A$ in $\Set$, the presheaf ${\overline{A}}$ in ${\Psh{L A}}$ is a sheaf (for the coherent coverage on the lattice ${L A}$).
\end{lemma}
\begin{proof}
To prove that ${\overline{A}}$ is a sheaf consider a cover ${(\eta a) \vee (\eta b) = \eta d}$ in ${L A}$.
Consider also a compatible family given by ${x \in \overline{A}(\eta a) =  A[a^{-1}]}$ and ${y \in \overline{A}(\eta b) = A[b^{-1}]}$. Compatibility means that ${x \cdot ((\eta a) \wedge (\eta b)) = y \cdot ((\eta a) \wedge (\eta b))}$. Now, since ${(\eta a) \wedge (\eta b) = \eta(a b)}$,  the previous equation simplifies to ${x\cdot \eta(a b) = y \cdot \eta(a  b) \in \overline{A} (\eta(a b)) = A[(a b)^{-1}]}$ as in the diagram below
$$\xymatrix{
A[a^{-1}] \ar[r] & A[(a b)^{-1}] & \ar[l] A[b^{-1}] \\
x \ar@{|->}[r] & = & y \ar@{|->}[l]
}$$
so, by Lemma~\ref{LemPBPO}, there exists  a unique ${z \in A[(a + b)^{-1}] = \overline{A}(\eta(a + b)) = \overline{A}(\eta d)}$ such that ${z\cdot (\eta a) = x}$ and ${z\cdot (\eta b) = y}$. In other words, the pushout-pullback lemma implies that ${\overline{A}}$ is a sheaf.
\end{proof}

   We now start the proof that the assignment that sends $A$ to $\overline{A}$ may be extended to a functor ${\iRig \rightarrow \frakI}$.

\begin{lemma}\label{LemRepresentationsAreSo}  For any integral rig $A$ in $\Set$, the pair ${(L A, \overline{A})}$ is an object in $\frakI$.
\end{lemma}
\begin{proof}
We need to prove that there is a local morphism of rigs ${\overline{A} \rightarrow \Lambda_{L A} = \Lambda}$.
We already know that $\overline{A}$ and $\Lambda$ are sheaves. So it is enough to prove that there is a reticulation ${\overline{A} \rightarrow  \Lambda}$ in ${\Psh{L A}}$. By Lemma~\ref{LemReticulationsInPresheafCats} it is enough to show that the composite
$$\xymatrix{\opCat{(L A)} \ar[r]^-{\overline{A}} & \iRig \ar[r]^-{L} & \dLat}$$
coincides with $\Lambda$. So let ${(\eta x)  \in L A}$ and observe that
\[L (\overline{A} (\eta x)) = L (A[x^{-1}]) = (L A)[(\eta x)^{-1}] = \below{(\eta x)} = \Lambda (\eta x) \]
 by Lemma~\ref{LemOverlineAndReticulation}.
\end{proof}

    Notice that if we apply ${\Gamma:\frakI \rightarrow \iRig}$ we obtain ${\Gamma (L A, \overline{A}) = \overline{A} \top = A[1^{-1}] \cong A}$. In this sense, every integral rig is the rig of global sections of a sheaf of really local integral rigs (over the spectral space determined by the reticulation of $A$). In Section~\ref{SecFuntorialityOfRepresentation} we will show that this representation is functorial. Before that, it seems relevant at this point to remark one special sort of example.

\begin{remark}\label{RemFiniteSpec}
    Let $A$ be an integral rig with reticulation ${A \rightarrow L A}$. Let ${P \rightarrow L A}$ be the subposet of (join-)irreducible elements in ${L A}$. If ${L A}$ is finite then, for general reasons (Lemma~{C2.2.21} in \cite{elephant}),  the inclusion ${P \rightarrow L A}$ induces an equivalence ${\Psh{P} \rightarrow \Shv(L A)}$ between the topos of presheaves on $P$ and the topos of sheaves on the lattice ${L A}$. The representing sheaf ${\overline{A}}$ in ${\Shv(L A)}$ corresponds to a presheaf ${B \in \Psh{P}}$ which must carry its associated structure of really local integral rig. By Lemma~\ref{LemReallyLocalRigsInPresheafToposes}, this means that  for every $p$ in $P$, ${B p}$ is a really local rig. In other words, every integral rig $A$ with finite reticulation may be represented as a functor ${F:\opCat{P} \rightarrow \iRig}$ for a finite poset $P$ and such that for every ${p \in P}$, ${F p}$ is really local. Details will be treated elsewhere.
\end{remark}

\section{The adjunction}
\label{SecFuntorialityOfRepresentation}

   In this section we show that the functor ${\Gamma:\frakI \rightarrow \iRig}$ has a fully faithful left adjoint.

\begin{lemma}\label{LemAdjunctionInDisguise}
Let $A$ be an integral rig in $\Set$ and let $R$ be an integral rig  in ${\Psh{L A}}$. For any ${g:A \rightarrow R\top}$ there exists at most one ${\phi:\overline{A} \rightarrow R}$ such that the following diagram
$$\xymatrix{
A \ar[rd]_-g \ar[r]^-{\simeq} & \overline{A} \top \ar[d]^-{\phi_{\top}} \\
  & R \top
}$$
commutes. Moreover, such $\phi$ exists if and only if for all ${x\in A}$, ${(g x) \cdot (\eta x) = 1 \in R (\eta x)}$.
\end{lemma}
\begin{proof}
Let ${\phi:\overline{A} \rightarrow R}$ be such that triangle in the statement commutes.
For any ${\eta x}$ in ${L A}$ the left square below
$$\xymatrix{
A \ar[d] \ar@/^1pc/[rr]^-g \ar[r]_-{\simeq} & \overline{A} \top  \ar[d] \ar[r]_-{\phi_{\top}} & R \top \ar[d] \\
A[x^{-1}] \ar[r]_-{=} & \overline{A} (\eta x) \ar[r]_-{\phi_{\eta x}} & R(\eta x)
}$$
commutes by definition of $\overline{A}$ and the right square above commutes by naturality of $\phi$. Then ${\phi_{\eta x}}$ is determined by ${\phi_{\top}}$ and the universal property of ${A\rightarrow A[x^{-1}]}$. This completes the proof of the first part of the statement.
Moreover, the rectangle also shows that for every ${x \in A}$, ${(g x) \cdot (\eta x) = 1}$.
For the converse assume that  ${(g x) \cdot (\eta x) = 1}$ holds for all ${x\in A}$. Then the universal property of ${A \rightarrow A[x^{-1}]}$ implies the existence of a unique ${\phi_{\eta x}:\overline{A}(\eta x) \rightarrow R(\eta x)}$ such that the rectangle above commutes. The universal property of localizations also implies that the collection ${(\phi_{\eta x} \mid (\eta x) \in L A)}$ is natural.
\end{proof}

   There is an analogous result for maps in $\frakI$.

\begin{lemma}\label{LemFrakImapsFromOverline}
Let $A$ be an object in $\iRig$ and ${(E, Y)}$ in $\frakI$. For every ${g:A \rightarrow Y\top}$ there exists at most one ${(f, \phi):(L A, \overline{A}) \rightarrow (E, Y)}$ in $\frakI$ such that the triangle on the left below
$$\xymatrix{
A \ar[rd]_-g \ar[r]^-{\simeq} & \overline{A} \top \ar[d]^-{\phi_{\top}} && A \ar[d]_-{\eta} \ar[r]^-g & Y \top \ar[d]^-{\chi_{\top}} \\
  & Y \top && L A \ar[r]_-f & E
}$$
commutes. Moreover, in this case, ${f:L A \rightarrow E}$ is the unique rig morphism such that the square on the right above commutes.
\end{lemma}
\begin{proof}
By Lemma~\ref{LemAdjunctionInDisguise} there exists at most one ${\phi:\overline{A} \rightarrow f_* Y}$ such that the triangle in the statement commutes. So it is enough to  prove that, assuming ${(f, \phi)}$ exists, ${f:L A \rightarrow E}$ makes the right square in the statement commute.
By hypothesis, the middle square below
$$\xymatrix{
A \ar[d]_-{\eta} \ar@/^1.5pc/[rrr]^-g \ar[r]_-{\simeq} & \overline{A} \top \ar[d]_-{\chi_{\top}} \ar[r]_-{\phi_{\top}} & (f_* Y)\top \ar[d]^-{(f_* \chi)_{\top}} \ar[r]_-{=} & Y \top \ar[d]^{\chi_{\top}}\\
L A \ar@/_1pc/[rrr]_-f \ar[r]^-{\simeq} & \Lambda_{L A} \top \ar[r]^-{f} & (f_* \Lambda_E) \top \ar[r]^-{\simeq} & E
}$$
commutes. Since the rest of the diagram commutes, the result follows.
\end{proof}

  We can now prove the main result of the paper.

\begin{theorem}\label{ThmMain}
The functor ${\Gamma: \frakI\rightarrow\iRig }$ has a full and faithful left adjoint.
\end{theorem}
\begin{proof}
   Fix an object $A$ in $\iRig$. By Lemma~\ref{LemRepresentationsAreSo}, ${(L A, \overline{A})}$ is an object in $\frakI$ and we can consider the (iso) map ${A \rightarrow A[1^{-1}] = \overline{A}\top = \Gamma (L A, \overline{A})}$. We claim that this map is universal from $A$ to $\Gamma$.
   To prove this let ${(E, Y)}$ in ${\frakI}$ and ${g:A \rightarrow \Gamma (E, Y) = Y \top}$ in $\iRig$.
     By Lemma~\ref{LemFrakImapsFromOverline}, there exists at most one map ${(f, \phi):(L A, \overline{A}) \rightarrow (E, Y)}$ in $\frakI$ such that the following diagram
$$\xymatrix{
A \ar[rd]_-{g} \ar[r]^-{\simeq} & \overline{A} \top \ar[d]^-{\phi_{\top}} & = & \Gamma(L A, \overline{A}) \ar[d]^-{\Gamma(f, \phi)}\\
 & Y \top & = & \Gamma(E, Y)
}$$
commutes. So, to complete the proof,  it is enough to construct one such ${(f, \phi)}$.
Again by Lemma~\ref{LemFrakImapsFromOverline},  ${f:L A \rightarrow E }$ is forced to be the unique rig morphism such that the square below
$$\xymatrix{
A \ar[d]_-{\eta} \ar[r]^-g & (f_* Y) \top = Y \top \ar[d]^-{\chi_{\top}} \\
L A \ar[r]_-f & E = \Lambda_{E}\top
}$$
commutes. By Lemma~\ref{LemDefAcutelyLocal}, for any ${y \in Y\top}$, ${y \cdot (\chi_{\top} y) = 1 \in Y(\chi_{\top} y)}$. In particular, for every ${x\in A}$, ${\chi_{\top} (g x) = f(\eta x)}$; so ${(g x) \cdot (f (\eta x)) = 1 \in Y(f(\eta x))}$
and hence, ${(g x) \cdot (\eta x) = 1}$  in ${(f_* Y)(\eta x)}$. By Lemma~\ref{LemAdjunctionInDisguise}, there exists a unique ${\phi:\overline{A} \rightarrow f_* Y}$ such that the following diagram
$$\xymatrix{
A \ar[rd]_-g \ar[r]^-{\simeq} & \overline{A} \top \ar[d]^-{\phi_{\top}} \\
 & (f_* Y) \top
}$$
commutes. We claim that ${(f, \phi)}$ is a map in $\frakI$ from ${(L A, \overline{A})}$ to ${(E, Y)}$. For this, we must check that the square on the left below
$$\xymatrix{
\overline{A} \ar[d]_-{\chi} \ar[r]^-{\phi} & f_* Y \ar[d]^-{f^*\chi} &&
   A \ar@/^1.5pc/[rrr]^-g \ar[d]_-{\eta} \ar[r]_-{\simeq} & \overline{A}\top \ar@{}[rd]|{(\ast)} \ar[d]_-{\chi_{\top}} \ar[r]^-{\phi_{\top}} & (f_* Y)\top \ar[d]^-{(f_*\chi)_{\top}} \ar[r]_-{=} & Y \top  \ar[d]^-{\chi_{\top}} \\
\Lambda_{L A} \ar[r]_-f & f_* \Lambda_{E} &&
  L A \ar@/_1.5pc/[rrr]_-f \ar[r]^-{\simeq} & \Lambda_{L A}\top \ar[r]_-f & (f_* \Lambda_{E})\top \ar[r]^-{=} & \Lambda_{E} \top \simeq E
}$$
commutes in ${\Shv(L A)}$. By Lemma~\ref{LemAdjunctionInDisguise}, it is enough to check that the square marked with ${(\ast)}$ on the right above commutes. Pre-composing with the iso ${A\rightarrow \overline{A}\top}$ we obtain the rectangle which commutes by definition of $f$.
 This completes the proof that ${A \rightarrow \Gamma(L A, \overline{A})}$ is universal from $A$ to $\Gamma$. Since this map is an iso the left adjoint is full and faithful.
\end{proof}

   The    left adjoint ${\iRig \rightarrow \frakI}$ sends ${A \in \iRig}$ to ${(L A, \overline{A})}$ and a map  ${h:A \rightarrow B}$ in $\iRig$ to the $\frakI$-map ${(L h, \psi)}$ where ${\psi:\overline{A} \rightarrow (L h)_* \overline{B}}$ is the unique rig map in ${\Shv(L A)}$ such that the following diagram
$$\xymatrix{
A \ar[d]_-h \ar[r]^-{\simeq} & \overline{A} \top \ar[d]^-{\psi_{\top}} \\
B \ar[r]_-{\simeq} & \overline{B}\top = ((L h)_* \overline{B}) \top
}$$
commutes.

   It is natural to ask if ${\Gamma:\frakI \rightarrow \iRig}$ is an equivalence.
   The answer seems to be `no'. The evidence comes from Lemma~\ref{LemRepresentationsAreSo}, which shows that the map ${\chi:\overline{A} \rightarrow \Lambda_{L A}}$ is the reticulation of ${\overline{A}}$ in ${\Shv(L A)}$. So we are led to consider the full subcategory of $\frakI$ determined by the representations ${(D, X)}$ such that ${\chi:X \rightarrow \Lambda_D}$ is a reticulation of $X$ in ${\Shv(D)}$. Is the restriction of ${\Gamma}$ to this subcategory an equivalence?

\section{The coextensive category of integral residuated rigs}
\label{SecIntegralResiduatedRigs}

   Let $A$ be a rig with canonical pre-order denoted by ${(A, \leq)}$. Any ${a \in A}$ determines a monotone functor ${a\cdot(\_): (A, \leq) \rightarrow (A, \leq)}$ and it is natural to consider cases where each of these functors has a right adjoint. In general these right adjoints are unique up to iso; but if addition in $A$ is idempotent then the canonical pre-order is a partial order and so the right adjoints are unique. In this case, the resulting category of structures can be equationally presented.

\begin{definition}\label{DefResiduatedRig}
A rig $A$ is called {\em residuated} if its addition is idempotent and for every ${a \in A}$, ${a\cdot (\_):A \rightarrow A}$ has a right adjoint (that will be denoted by ${a\multimap (\_)}$). If $A$ and $B$ are residuated rigs then a {\em morphism} ${f:A \rightarrow B}$ is a map ${f:A \rightarrow B}$ between the underlying rigs such that for every ${x, y \in A}$, ${f(x\multimap y) = (f x) \multimap (f y)}$.
\end{definition}

   Let ${\rRig}$ denote the algebraic category of residuated rigs in $\Set$ and let ${\riRig \rightarrow \rRig}$ the full subcategory of those residuated rigs that are also integral.
   For brevity,  residuated integral rigs will be referred to as {\em ri-rigs}.

\begin{lemma}\label{LemLocalizationsLiftToRirigs}
If $A$ is a ri-rig and ${F \rightarrow A}$ is a multiplicative submonoid then the universal ${A \rightarrow A[F^{-1}]}$ in $\iRig$ is also the universal solution to inverting ${F}$ in $\riRig$.
\end{lemma}
\begin{proof}
It is enough to check that the congruence built in Lemma~\ref{LemSubMultiplicationsCanBeCollapsed} is a congruence of residuated rigs.
Recall that we defined the congruence $\equiv_F$ on $A$ as that induced by the pre-order ${\mid_F}$ which, in turn, was defined by declaring that ${x \mid_F y}$ if and only if there exists a ${w \in F}$ such that ${w x \leq y}$.
So, to prove the present result, it is enough to check that if ${x\equiv_F x'}$ and ${y \equiv_F y'}$ then ${(x \multimap y) \mid_F (x' \multimap y')}$.
Now, if ${x'\mid_F x}$ and ${y \mid_F y'}$ then there are ${u', v \in F}$ such that ${u' x' \leq x}$ and ${v y \leq y'}$. The product ${v u'}$ is in $F$ and, as ${(x \multimap y) \leq (u' x' \multimap y)}$ we have that
\[ v u' x'  (x \multimap y) \leq v u' x'  (u' x' \multimap y) \leq v y \leq y' \]
so ${(v u') (x \multimap y) \leq (x' \multimap y')}$.
\end{proof}

   Roughly speaking, localizations in $\riRig$ may be calculated as in $\iRig$.
   This  allows us to lift many constructions from the latter category to the former.
   In fact, we are going to lift all of Theorem~\ref{ThmMain}.

\begin{corollary}\label{CorRiRigIsCoextensive}
The category $\riRig$ is coextensive.
\end{corollary}
\begin{proof}
The initial integral rig $2$ is a Heyting algebra and it is also initial in $\riRig$.
Lemma~\ref{LemLocalizationsLiftToRirigs} implies that the proof of Proposition~\ref{PropIRigIsCoextensive} lifts to a proof that $\riRig$ is coextensive.
\end{proof}

  The pushout-pullback lemma is essentially about localizations and pullbacks; and these are calculated as in $\iRig$.

\begin{corollary}\label{CorPOPB}
If $A$ is in $\riRig$ then for every ${a, b \in A}$ the canonical maps make the following diagram commute
$$\xymatrix{
A[(a + b)^{-1}] \ar[d] \ar[r] & A[b^{-1}] \ar[d] \\
A[a^{-1}] \ar[r] & A[(a b)^{-1}]
}$$
and, moreover, the square is both a pullback and a pushout.
\end{corollary}

\newcommand{\frakrI}{\mathfrak{rI}}

   Let ${\frakrI}$ be the subcategory of ${\frakI}$ consisting of the objects ${(D, X)}$ such that ${X}$ is residuated in ${\Shv(D)}$ and whose maps ${(f, \phi):(D, X) \rightarrow (E, Y)}$ such that ${\phi:X \rightarrow f_* Y}$ is a morphism of residuated lattices in ${
\Shv(D)}$.
   It follows that  we have a functor ${\Gamma:\frakrI \rightarrow \riRig}$ such that the following diagram
$$\xymatrix{
\frakrI \ar[d] \ar[r]^-{\Gamma} & \riRig \ar[d] \\
\frakI \ar[r]_-{\Gamma} & \iRig
}$$
commutes, where the vertical functors are the obvious forgetful functors.

\begin{corollary}\label{CorMain}
The functor ${\Gamma:\frakrI \rightarrow \riRig}$ has a full and faithful left adjoint.
\end{corollary}
\begin{proof}
For any $A$ in $\riRig$, the integral rig $\overline{A}$ is still a sheaf in ${\Shv(L A)}$ and it is residuated by Lemma~\ref{LemLocalizationsLiftToRirigs}. (We stress that ${L A}$ is the reticulation of $A$ as an integral rig. The residuated structure plays no role in the construction of the topos where the representing really local algebra lives.)
So ${(L A, \overline{A})}$ is in ${\frakrI}$.
Also, given any $A$ in $\riRig$, ${(E, Y)}$ in ${\frakrI}$ and ${g:A \rightarrow \Gamma(E, Y) = Y \top}$ in ${\riRig}$, the map ${(f, \phi):(L A, \overline{A}) \rightarrow (E, Y)}$ is in $\frakrI$ because $\phi$ is built using universal properties of localizations and so it must be a morphism of residuated lattices.
\end{proof}

   So every  ${A \in \riRig}$ is the rig of global sections of a sheaf $\overline{A}$ of really local integral residuated rigs.

\begin{definition}\label{DefPrelinear}
A residuated rig $A$ is {\em pre-linear} if the equation ${(a\multimap b) + (b\multimap a) = 1}$ holds for every ${a, b \in A}$.
\end{definition}

   The fibers of our representation result now start to look familiar.

\begin{lemma}\label{LemReallyLocalPrelinear}
Let $A$ be a non-trivial integral rig. If the canonical pre-order of $A$ is total then $A$ is really local. If $A$ is also pre-linear then the converse holds.
\end{lemma}
\begin{proof}
If $A$ is pre-linear and really local then, for any ${a, b \in A}$, ${1 = (a \multimap b)}$ or ${1 = (b\multimap a)}$.
\end{proof}

\newcommand{\pliRig}{\mathbf{pliRig}}
\newcommand{\frakplI}{\mathfrak{plI}}
\newcommand{\frakmvI}{\mathfrak{plI}}

   Let ${\pliRig \rightarrow \riRig}$ be the variety of pre-linear integral rigs. Of course, localizations in $\pliRig$ are calculated as in $\riRig$.
   Let ${\frakplI}$ be the full subcategory of ${\frakrI}$ determined by the objects ${(D, X)}$ such that $X$ in pre-linear in ${\Shv(D)}$. The functor ${R:\frakrI \rightarrow  \riRig}$ restricts to a functor ${R:\frakplI \rightarrow \pliRig}$ and Corollary~\ref{CorMain} restricts to the following result.

\begin{corollary}\label{CorPreLinear}
The functor ${R:\frakplI \rightarrow \pliRig}$ has a full and faithful left adjoint.
\end{corollary}

   That is, every pre-linear integral rig $A$ is, functorially, the rig of global sections of a sheaf $\overline{A}$ in ${\Shv(L A)}$ with totally ordered fibers. Also, continuing with the idea outlined in Remark~\ref{RemFiniteSpec}, we may conclude that every pre-linear integral rig with finite reticulation may be represented as a functor from a finite poset to the category of totally ordered integral residuated rig.

   Further restrictions on the algebraic theory provide other new representation results in terms of sheaves with totally ordered fibers. For example, one may restrict to pre-linear integral  rigs satisfying ${\neg\neg x = x}$ to obtain a representation result in terms of totally order fibers. (Here, as usual,  ${\neg x = x\multimap 0}$.)

   Another case worth mentioning is that of MV-algebras. If one restricts to pre-linear integral residuated rigs satisfying ${\neg\neg x = x}$ and the Wajsberg condition then we obtain a variant of the Dubuc-Poveda representation theorem, involving only spectral maps between the spectral bases. We give some of the details in Section~\ref{SecMV} below.

\section{Corollaries in terms of local homeos}
\label{SecLH}

\newcommand{\spec}{Spec}
\newcommand{\ccompactd}{\overline{\mathcal{K}}(spec(D))}
\newcommand{\compactd}{\mathcal{K}(Spec(D))}
\newcommand{\ccompactx}{\overline{\mathcal{K}}(X)}
\newcommand{\compactx}{\mathcal{K}(X)}
\newcommand{\openx}{\mathcal{O}(X)}

In this section we briefly explain how our results may be expressed in terms of local homeos with algebraic structure on the fibers.
    It is a classical result that for any topological space $X$, the category ${\LH/X}$ of local homeos over $X$ is equivalent to the topos ${\Shv(X)}$  of sheaves over the same space (see Section~{II.6} in \cite{maclane2}). The equivalence ${\Shv(X) \rightarrow \LH/X}$ sends a sheaf ${P:\opCat{\openx} \rightarrow \Set}$ to the {\em bundle of germs} of $P$ defined as follows.
For each ${x\in X}$, let ${P_x = \varinjlim_{x \in U} P U}$ where the colimit is taken over the poset of open neighborhoods of $x$ (ordered by reverse inclusion). The family of ${P_x}$'s determines a function ${\pi:\sum_{x\in X} P_x \rightarrow X}$. Also, each ${s\in P U}$  determines an obvious function ${\dot{s}:U \rightarrow \sum_{x\in X} P_x}$ such that ${\pi \dot{s}:U \rightarrow X}$ is the inclusion ${U \rightarrow X}$. The set ${\sum_{x\in X} P_x}$ is topologized by taking as a base of opens all the images of the functions $\dot{s}$. This topology makes $\pi$ into a local homeo, the above mentioned bundle of germs.

    Any basis $B$ for the topology of $X$ may be considered as a subposet  ${B\rightarrow \openx}$. The usual Grothendieck topology on ${\openx}$ restricts along ${B \rightarrow \openx}$ and the resulting morphism of sites determines an equivalence ${\Shv(B) \rightarrow \Shv(X)}$; see Theorem~{II.1.3} in \cite{maclane2}.
   The composite equivalence ${\Shv(B) \rightarrow \Shv(X) \rightarrow \LH/X}$ is very similar to the previous one because, by finality (in the sense of Section~{IX.3} of \cite{maclane}), the colimit  ${P_x = \varinjlim_{x \in U} P U}$ may be calculated using only basic open sets.

    We now concentrate on spectral spaces, following \cite{Simmons80, Johnstone1982}.
    The {\em spectrum} of a distributive lattice $D$ is the topological space ${\sigma D}$ whose points are the lattice morphisms ${D\rightarrow 2}$ (where $2$ denotes the totally ordered lattice ${ \{\bot < \top \} }$) and whose topology has, as a basis, the  subsets ${\sigma(a) \subseteq \sigma D}$ (with ${a\in D}$) defined by ${\sigma(a) = \{ p \in \sigma D \mid p a = \top \in 2 \} \subseteq \sigma D}$. In this way, we may identify $D$ with the basis of its spectrum and obtain an equivalence ${\Shv(D) \rightarrow \LH/\sigma D}$. It assigns to each sheaf ${P:\opCat{D} \rightarrow \Set}$ the local homeo  whose fiber $P_p$ over the point ${p:D \rightarrow 2}$ in ${\sigma D}$ is
\[ P_p =  \varinjlim_{p \in  \sigma(a)} P a = \varinjlim_{p a = \top} P a \]
as follows from the more general descriptions above. In other words, the fiber is the colimit of the functor
$$\xymatrix{ \opCat{(p^{-1} \top)} \ar[r] & \opCat{D}  \ar[r]^-P  & \Set }$$
where ${p^{-1} \top \subseteq D}$ is considered as a poset inclusion.

   Now let us consider an integral rig $A$ and its reticulation ${\eta:A \rightarrow L A}$. Precomposition with $\eta$ gives a bijection between the points of the spectrum ${\sigma(L A)}$ and the set ${\iRig(A, 2)}$. Indeed, it is just the bijection ${\dLat(L A, 2) \rightarrow \iRig(A, 2)}$.
Also, for each ${x \in A}$, the basic open ${\sigma(\eta x) \subseteq \sigma(L A)}$ is the subset of those ${p:L A \rightarrow 2}$ such that ${p (\eta x) = \top}$. So ${\sigma(\eta x)}$ may be identified with the subset of ${\iRig(A, 2)}$ given by those ${p:A \rightarrow 2}$ such that ${p x = \top}$. For this reason it is convenient to define the {\em spectrum} of the integral rig $A$ as the topological space ${\sigma A}$  whose set of points is ${\iRig(A, 2)}$ and whose topology is determined by the basic open sets of the form
${\sigma x = \{p:A \rightarrow 2 \mid p x = \top \}}$. We stress that for an arbitrary integral rig $A$ there may be different ${x, y \in A}$ such that ${\sigma(x) = \sigma(y)}$. On the other hand, if the integral rig $A$ is a distributive lattice then ${\sigma A}$ is the spectrum of the lattice $A$ as we originally defined it.

   Putting things together we obtain an equivalence ${\Shv(L A) \rightarrow \LH/\sigma A}$. It sends a sheaf ${P \in \Shv(L A)}$ to a local homeo over ${\sigma A}$ whose  fiber ${P_p}$ over a point ${p:A \rightarrow 2}$ in ${\sigma A}$ may be described as
\[ P_p =   \varinjlim_{p x = \top} P (\eta x) \]
where $x$ ranges over the elements of $A$. In other words, the fiber is the colimit of the functor
$$\xymatrix{ \opCat{(p^{-1} \top)} \ar[r] & \opCat{A}  \ar[r]^-{\eta} & \opCat{(L A)} \ar[r]^-P  & \Set }$$
where ${p^{-1} \top \subseteq A}$ is considered as a poset inclusion.

   Consider now the representing sheaf ${\overline{A} \in \Shv(L A)}$ of $A$. The fiber over the point ${p:A \rightarrow 2}$ is
\[  (\overline{A})_p =   \varinjlim_{p x = \top} \overline{A} (\eta x)  = \varinjlim_{p x = \top}  A[x^{-1}] = A[ (p^{-1} \top)^{-1}] \]
using Lemma~\ref{LemColomitFilters} in the last step. In other words, when we look at the representing sheaf ${\overline{A} \in \Shv(L A)}$ as a local homeo over ${\sigma A}$, then the fiber over a point ${p:A \rightarrow 2}$ in ${\sigma A}$ is the localization of $A$ at the multiplicative submonoid ${p^{-1}\top \rightarrow A}$.

   Returning to topological spaces, every point ${x:1 \rightarrow X}$ of a space $X$ determines a geometric morphism ${\Set \rightarrow \LH/X}$ whose inverse image ${\LH/X \rightarrow \Set}$ sends a local homeo to the corresponding fiber over $x$.
Since inverse images of geometric morphisms preserve finite limits and colimits, they preserverve really local integral rigs.
So, as a byproduct of our main result, we obtain the following.

\begin{corollary}
Every integral rig may be represented as the algebra of global sections of a local homeo (over the spectral space ${\sigma A}$) whose fibers are really local integral rigs.
\end{corollary}

   We can further restrict the result as follows.

\begin{corollary}
Every integral rig is a subdirect product of really local integral rigs.
\end{corollary}
\begin{proof}
Let $A$ be an integral rig and let ${\sum_{p\in \sigma A} (\overline{A})_p \rightarrow \sigma A}$ be  its representing local homeo. We know that its algebra of continuous sections is iso to $A$; but the set of all (non continuous) sections is isomorphic to ${\prod_{p\in \sigma A} (\overline{A})_p}$. The resulting inclusion ${A \rightarrow \prod_{p\in\sigma A} (\overline{A})_p}$ is determined by the localizations ${A \rightarrow (\overline{A})_p = A[(p^{-1} \top)^{-1}]}$.
\end{proof}

   Much as in the case of the main result, the corollaries lift to residuated rigs. So, for example, we obtain the following.

\begin{corollary}
Every integral pre-linear  rig is a subdirect product of totally ordered integral residuated rigs.
\end{corollary}

\section{The representation of MV-algebras}
\label{SecMV}

\newcommand{\mv}{\mathbf{MV}}
\newcommand{\mvRig}{\mathbf{mvRig}}

   An {\em MV-algebra} is usually defined as a structure ${(A, \oplus, 0, \neg)}$ such that ${(A, \oplus, 0)}$ is a commutative monoid and ${\neg:A \rightarrow A}$ is a function such that the following equations hold:
\begin{enumerate}
\item ${\neg\neg x = x}$,
\item ${x\oplus (\neg 0) = \neg 0}$,
\item ${\neg(\neg x \oplus y) \oplus y = \neg(\neg y \oplus x) \oplus x}$.
\end{enumerate}
This determines an algebraic category that we denote by ${\mv \rightarrow \Set}$.
It is well-known that every MV-algebra has an associated lattice structure and a residuated binary operation (see Lemma~{1.1.4(iii)} in \cite{Cignolietal2000} where the residuated operation is denoted by ${\odot}$). In order to make this more explicit and relate it to our work we introduce the following.

\begin{definition} An {\em MV-rig} is an integral residuated rig ${(A, \cdot, 1, +, 0, \multimap)}$ such that the following (Wajsberg) condition:
\[  (x \multimap y) \multimap y = (y \multimap x) \multimap x \]
holds.
\end{definition}

   This definition also determines an algebraic category ${\mvRig \rightarrow \Set}$ and there is a functor ${\mvRig \rightarrow \mv}$ that sends the MV-rig ${(A, \cdot, 1, +, 0, \multimap)}$ to the MV-algebra ${(A, \oplus, \neg)}$ where ${\neg x = x \multimap 0}$ and ${x\oplus y = \neg ((\neg x) \cdot (\neg y))}$. Moreover, it is surely at least folklore that this functor is an equivalence of algebraic categories; in fact, an isomorphism. Its inverse ${\mv \rightarrow \mvRig}$ sends an MV-algebra ${(A, \oplus, \neg)}$ to the MV-rig with the same underlying set and operations defined as follows
   \begin{center}
\begin{tabular}{cc}
$x+y=\neg (\neg x \oplus y) \oplus y$ & $1=\neg 0$
\\
$x\cdot y=\neg (\neg x \oplus \neg y)$ & $x\multimap y=\neg x \oplus y$
\\
\end{tabular}
\end{center}
for every ${x, y \in A}$. It is relevant to recall that for every MV-algebra $A$, the semilattice ${(A, +, 0)}$ extends to lattice structure with  ${u \wedge v = \neg((\neg u) + (\neg v))}$ for every ${u, v \in A}$.

   Every MV-rig is pre-linear so our results imply that every MV-rig is the algebra of global sections of a local homeo whose fibers are totally ordered MV-rigs. In the rest of the section we show that this is essentially the Dubuc-Poveda representation \cite{DubucPoveda2010}.

   From now on, until the end of the section, let ${M = (A, \oplus, \neg)}$ be an MV-algebra and let ${R = (A,  \cdot, 1, +, 0, \multimap)}$ be the associated MV-rig.
   The {\em natural order} of $M$ may defined by declaring that ${x\leq_n y}$ if ${x\cdot (\neg y)  = 0}$. (See Lemma~{1.1.2} in \cite{Cignolietal2000}.) This is equivalent to ${x\leq \neg(\neg y) = y}$ where ${\leq}$ is the canonical order of the rig $R$. Therefore, ${x \leq_n y}$ in $M$ if and only if ${x\leq y}$ in $R$. In other words, the natural order of $A$ as an MV-algebra coincides with the canonical order of $A$ as a rig. So, from now, we drop the ${\leq_n}$ notation.

   An {\em ideal} of the MV-algebra $M$ is a submonoid ${(I, \oplus, 0) \rightarrow  (A, \oplus, 0)}$ such that for every ${x \in I}$ and ${y \in A}$, ${y \leq x}$ implies ${y \in I}$.
   Such an ideal is called {\em proper} if ${1\not\in I}$. It is called {\em prime} if it is proper and also, for any ${x, y\in A}$, either ${(x \cdot (\neg y)) \in I}$ or ${(y \cdot (\neg x)) \in I}$. We will use the following characterization: a proper ideal ${I \rightarrow M}$ is prime if and only if, for every ${x, y \in M}$, ${x\wedge y \in I}$ implies ${x\in I}$ or ${y\in I}$. See, for example, 1.3 in \cite{DubucPoveda2010}.

\begin{lemma}\label{LemPrimeIdealsAs2vauedMaps}
For every ${p:R \rightarrow 2}$ in $\iRig$, the subset ${ \neg(p^{-1} \top) = \{ \neg x \mid p x = \top \} \rightarrow A}$ is a prime ideal of $M$.
\end{lemma}
\begin{proof}
For brevity let ${I = \neg(p^{-1} \top)}$.
Since ${p 1 = \top}$, ${0 = \neg 1 \in I}$.
Also, if ${1 \in I}$ then there exists ${x\in A}$ such that ${p x = \top}$ and ${\neg x = 1}$. Then ${x = 0}$ and so ${p 0 = \top}$, absurd.
To prove that $I$ is closed under $\oplus$, let ${x, y \in A}$ be such that ${p x = \top = p y}$. We need to check that ${(\neg x) \oplus (\neg y) = \neg(x\cdot y) \in I}$; but this holds because ${p(x\cdot y) = (p x) \wedge (p y) = \top}$.

   Assume now that ${x = p u = \top}$ and that ${y \leq \neg u}$. Then ${u \leq \neg y}$, so ${p(\neg y) = \top}$ and hence ${y \in I}$.
To complete the proof recall that, in any MV-algebra, ${\neg(x\cdot (\neg y)) + \neg((\neg x)\cdot y) = 1}$. (See, for example, Proposition~{1.1.7} in \cite{Cignolietal2000} which shows that ${(x \cdot (\neg y)) \wedge (y \cdot (\neg x)) = 0}$.) Then we have that ${p( \neg(x\cdot (\neg y))) \vee p(\neg((\neg x)\cdot y) = \top}$ and so, either ${p( \neg(x\cdot (\neg y))) = \top}$ or ${p(\neg((\neg x)\cdot y)) = \top}$. Therefore, ${x\cdot (\neg y) \in I}$ or ${(\neg x)\cdot y \in I}$.
\end{proof}

  Those familiar with Wajsberg algebras may recognize ${p^{-1}\top}$ as an implicative filter. See Section~{4.2} in \cite{Cignolietal2000}.

\begin{proposition}\label{PropPrimeIdeals}
For any prime ideal ${I \rightarrow M}$ there exists a unique ${p:R \rightarrow 2}$ in $\iRig$ such that ${I = \neg(p^{-1} \top)}$.
\end{proposition}
\begin{proof}
Let ${p, q:R \rightarrow 2}$ in ${\iRig}$ be such that the prime ideals ${\neg(p^{-1} \top)}$ and ${\neg(q^{-1} \top)}$ coincide. We want to show that ${p = q}$. For this, it is enough to prove that for every ${x \in R}$, ${p x = \top}$ if and only if  ${q x = \top}$. It suffices to establish only one implication so let ${p x = \top}$. Then  ${\neg x \in \neg(p^{-1} \top) = \neg(q^{-1} \top)}$. So there exists a ${y \in R}$ such that ${q y = \top}$ and ${\neg y = \neg x}$. Then ${x = y}$ and ${q x = \top}$.

   To prove the existence part of the statement let ${p:A \rightarrow 2}$ be the characteristic map (in $\Set$) of the subset ${\neg I = \{\neg x \mid x\in I\} \subseteq A}$. In other words, for every ${x\in A}$, ${p x = \top}$ if and only if ${\neg x \in I}$. Clearly, ${\neg(p^{-1} \top) = \neg(\neg I) = I}$, so it only remains to show that  $p$ underlies a rig morphism ${R \rightarrow 2}$. Since ${0\in I}$, ${p 1 = \top}$. Since $I$ is proper, ${p 0 = \bot}$. To prove that $p$ preserves the additive structure notice that ${p(x + y) = \top}$ if and only if ${\neg(x + y) = (\neg x) \wedge (\neg y) \in I}$. Since $I$ is prime, the previous statement holds if and only if ${\neg x \in I}$ or ${\neg y \in I}$, but this is equivalent to ${(p x) \vee (p y) = \top}$.

   Finally, to prove that ${p:A \rightarrow 2}$ preserves multiplication, observe that ${p(x\cdot y) = \top}$ if and only if ${\neg(x\cdot y) \in I}$ if and only if ${(\neg x) \oplus (\neg y) \in I}$ if and only if ${\neg x \in I}$ and ${\neg y \in I}$ if and only if ${(p x)\wedge (p y) = \top}$.
\end{proof}

   We now compare the spectrum ${\sigma R}$ of the integral rig $R$ (in the sense of Section~\ref{SecLH}) with the space ${Z_M}$ for the  MV-algebra $M$ considered in \cite{DubucPoveda2010}.  The points  of ${Z_M}$ are exactly the prime ideals of $M$. The topology on ${Z_M}$ is determined by the basic opens of the form ${W_a = \{ I \in Z_M \mid a \in I \}}$.

\begin{corollary}
The bijection of Proposition~\ref{PropPrimeIdeals} underlies an iso between the spaces ${\sigma R}$ and ${Z_M}$.
\end{corollary}
\begin{proof}
   Recall (Section~\ref{SecLH}) that the  spectrum ${\sigma R}$ of the integral rig $R$ has ${\iRig(R, 2)}$ as its set of points. So Proposition~\ref{PropPrimeIdeals} induces a bijection ${\varphi:\sigma R \rightarrow Z_M}$. We claim that it is continuous. To see this recall that the topology of ${\sigma R}$ is given by the basic sets of the form ${\sigma(x) = \{ p:R \rightarrow 2 \mid p x = \top \}}$ with ${x \in A}$.
   To prove that $\varphi$ is continuous consider a basic open ${W_a}$ of $Z_M$ and calculate
\[ p \in \varphi^{-1} W_a \quad \Leftrightarrow \quad  \varphi p \in W_a \quad \Leftrightarrow \quad  a\in \neg (p^{-1}\top) \quad \Leftrightarrow \quad p(\neg a) = \top \]
to conclude that ${\varphi^{-1} W_a = \sigma(\neg a)}$. To prove that this continuous bijection is an iso it is enough to prove that it is open. Let ${\sigma(x)}$ be a basic open in ${\sigma R}$ and calculate
\[ \varphi(\sigma(x)) = \{ \neg (p^{-1} \top) \mid p x =\top \} = \{ I \mid \neg x \in I \} = W_{\neg x} \]
to complete the proof.
\end{proof}

   Any continuous function ${f:X \rightarrow Y}$ between topological spaces induces a geometric morphism ${\LH/X \rightarrow \LH/Y}$ between the corresponding toposes of sheaves. The inverse image ${f^*:\LH/Y \rightarrow \LH/X}$ sends a local homeo over $Y$ to its pullback along ${f:X \rightarrow Y}$ (calculated as in the category of topological spaces). (See Theorem~{II.9.2} in \cite{maclane2}.) In particular, for any local homeo ${E \rightarrow Y}$, the fiber ${(f^* E)_x}$ over ${x\in X}$ may be identified with the fiber ${E_{f x}}$ of ${E \rightarrow Y}$ over ${f x}$.

   In particular, the iso ${\varphi:\sigma R \rightarrow Z_M}$ induces an equivalence ${\varphi^*:\LH/Z_M \rightarrow \LH/\sigma R}$. The Dubuc-Poveda representation of the MV-algebra $M$ consists of a space ${E_M}$ whose underlying set is the coproduct ${\sum_{I\in Z_M} M/I}$ where ${M/I}$ is the quotient of $M$ by the prime ideal $I$. The topology on ${E_M}$ is such that the indexing ${E_M \rightarrow Z_M}$ is a local homeo whose algebra of global sections is isomorphic to $M$. Now calculate
\[ (\varphi^* E_M)_p = (E_M)_{\varphi p} =  M/(\varphi p) \]
and observe that, since ${M \rightarrow M/(\varphi p)}$ is the universal way of forcing the ideal ${\phi p \rightarrow M}$ to be $0$, then it also has the universal way of inverting
${\neg (\varphi p) = \neg (\neg (p^{-1} \top)) = p^{-1} \top}$. Altogether, ${(\varphi^* E_M)_p = R[(p^{-1}\top)^{-1}] = (\overline{R})_p}$ where ${(\overline{R})_p}$ is the fiber of our representation as explained in Section~\ref{SecLH}. Roughly speaking, the representing local homeos in \cite{DubucPoveda2010} are essentially the same as ours.

   Alternatively, it is proved in 2.3 of \cite{DubucPoveda2010} that the local homeo ${E_M \rightarrow Z_M}$ is the result of applying the Godement construction to the presheaf that sends ${W_a}$ to ${M/(a)}$ where ${(a)}$ is the ideal generated by ${a \in M}$. Again, ${M/(a)}$ is, up to the equivalence between MV-algebras and MV-rigs, the same as ${R[(\neg a)^{-1}]}$. So we see again that our construction, restricted to MV-algebras, produces essentially the same result as that by Dubuc and Poveda.

  The main difference with the main result in \cite{DubucPoveda2010} is that our category of representing objects is much smaller. Notice, in particular, that we do not deal with arbitrary continuous maps but only with spectral ones. We leave a more detailed comparison for the reader.


\begin{thebibliography}{10}

\bibitem{BrezuleanuDiaconescu69}
A.~{Brezuleanu} and R.~{Diaconescu}.
\newblock {La duale de la categorie Ld(0,1).}
\newblock {\em {C. R. Acad. Sci., Paris, S\'er. A}}, 269:499--502, 1969.

\bibitem{CarboniLackWalters}
A.~Carboni, S.~Lack, and R.~F.~C. Walters.
\newblock Introduction to extensive and distributive categories.
\newblock {\em Journal of Pure and Applied Algebra}, 84:145--158, 1993.

\bibitem{Cole}
J.~C. Cole.
\newblock The bicategory of topoi, and spectra.
\newblock Unpublished.

\bibitem{Coste79}
M.~{Coste}.
\newblock {Localisation, spectra and sheaf representation.}
\newblock {Applications of sheaves, Proc. Res. Symp., Durham 1977, Lect. Notes
  Math. 753, 212-238.}, 1979.

\bibitem{DaunsHofmann66}
J.~{Dauns} and K.~H. {Hofmann}.
\newblock {The representation of biregular rings by sheaves.}
\newblock {\em {Math. Z.}}, 91:103--123, 1966.

\bibitem{Davey73}
B.~A. {Davey}.
\newblock {Sheaf spaces and sheaves of universal algebras.}
\newblock {\em {Math. Z.}}, 134:275--290, 1973.

\bibitem{DNG}
A.~Di~Nola and G.~Georgescu.
\newblock Grothendieck-like duality for {H}eyting algebras.
\newblock {\em Algebra Universalis}, 47(3):215--221, 2002.

\bibitem{DiNolaLeustean2003}
A.~{Di Nola} and L.~{Leu\c stean}.
\newblock {Compact representations of BL-algebras.}
\newblock {\em {Arch. Math. Logic}}, 42(8):737--761, 2003.

\bibitem{DubucErratum2010}
E.~J. Dubuc.
\newblock Erratum to ``{R}epresentation theory of {MV}-algebras'' [{A}nn.
  {P}ure {A}ppl. {L}ogic 161 (8) (2010)] [mr2629504].
\newblock {\em Ann. Pure Appl. Logic}, 163(9):1358, 2012.

\bibitem{DubucPoveda2010}
E.~J. Dubuc and Y.~A. Poveda.
\newblock Representation theory of {MV}-algebras.
\newblock {\em Ann. Pure Appl. Logic}, 161(8):1024--1046, 2010.

\bibitem{FL}
A.~R. Ferraioli and A.~Lettieri.
\newblock Representations of {MV}-algebras by sheaves.
\newblock {\em Mathematical Logic Quarterly}, 57(1):27--43, 2011.

\bibitem{FG}
A.~Filipoiu and G.~Georgescu.
\newblock Compact and pierce representations of {MV}-algebras.
\newblock {\em Revue Roumaine de Math{\'e}matiques Pures et Appliqu{\'e}es},
  40(7):599--618, 1995.

\bibitem{ega1}
A.~Grothendieck and J.~Dieudonn\'e.
\newblock {\'E}l\'ements de {G}\'eom\'etrie alg\'ebrique.
\newblock {\em Publications Math\'ematiques de l'Institut des Hautes \'Etudes
  Scientifiques}, 4(1):5--214, 1960.

\bibitem{Johnstone77a}
P.~T. Johnstone.
\newblock Rings, fields and spectra.
\newblock {\em Journal of {A}lgebra}, 49:238--260, 1977.

\bibitem{Johnstone1982}
P.~T. Johnstone.
\newblock {\em Stone Spaces}, volume~3 of {\em Cambridge Studies in advanced
  mathematics}.
\newblock Cambridge University Press, Cambridge, 1982.

\bibitem{elephant}
P.~T. Johnstone.
\newblock {\em Sketches of an elephant: a topos theory compendium}, volume
  43-44 of {\em Oxford Logic Guides}.
\newblock The Clarendon Press Oxford University Press, New York, 2002.

\bibitem{Lawvere91}
F.~W. Lawvere.
\newblock Some thoughts on the future of category theory.
\newblock In {\em Proceedings of {C}ategory {T}heory 1990, Como, Italy}, volume
  1488 of {\em Lecture notes in mathematics}, pages 1--13. Springer-Verlag,
  1991.

\bibitem{LawvereEmailCoquand}
F.~W. Lawvere.
\newblock Grothendieck's 1973 {B}uffalo {C}olloquium.
\newblock Email to the {\em categories} list, March 2003.

\bibitem{Lawvere08}
F.~W. Lawvere.
\newblock Core varieties, extensivity, and rig geometry.
\newblock {\em Theory Appl. Categ.}, 20(14):497--503, 2008.

\bibitem{maclane}
S.~{Mac Lane}.
\newblock {\em Categories for the Working Mathematician}.
\newblock Graduate Texts in Mathematics. Springer Verlag, 1971.

\bibitem{maclane2}
S.~{Mac Lane} and I.~Moerdijk.
\newblock {\em Sheaves in Geometry and Logic: a First Introduction to Topos
  Theory}.
\newblock Universitext. Springer Verlag, 1992.

\bibitem{Pierce67}
R.S. {Pierce}.
\newblock {Modules over commutative regular rings.}
\newblock {\em {Mem. Am. Math. Soc.}}, 70:112, 1967.

\bibitem{Cignolietal2000}
I.~M. L.~D'Ottaviano R.~L. O.~Cignoli and D.~Mundici.
\newblock {\em Algebraic Foundations of Many - Valued Reasoning}, volume~7 of
  {\em Trends in Logic}.
\newblock Springer Science + Business Media, 2000.

\bibitem{Schanuel91}
S.~H. Schanuel.
\newblock {Negative sets have Euler characteristic and dimension.}
\newblock {Category theory, Proc. Int. Conf., Como/Italy 1990, Lect. Notes
  Math. 1488, 379-385 (1991).}

\bibitem{Simmons80}
H.~{Simmons}.
\newblock {Reticulated rings.}
\newblock {\em {J. Algebra}}, 66:169--192, 1980.

\end{thebibliography}
\end{document}